\documentclass[11pt]{article}

\usepackage[margin=1in]{geometry}
\usepackage{amsmath,amssymb,amsthm,mathtools}
\usepackage{bm}
\usepackage{xcolor}
\usepackage{enumitem}
\usepackage{graphicx}
\usepackage{hyperref}
\usepackage{booktabs}

\hypersetup{
  colorlinks=true,
  linkcolor=blue,
  citecolor=blue,
  urlcolor=blue
}

\newtheorem{theorem}{Theorem}[section]
\newtheorem{proposition}[theorem]{Proposition}
\newtheorem{lemma}[theorem]{Lemma}
\newtheorem{corollary}[theorem]{Corollary}
\newtheorem{remark}[theorem]{Remark}
\numberwithin{equation}{section}

\newcommand{\E}{\mathbb{E}}
\newcommand{\Pbb}{\mathbb{P}}
\newcommand{\R}{\mathbb{R}}

\newcommand{\1}{\mathbf 1}
\newcommand{\sech}{\operatorname{sech}}

\newcommand{\cK}{\mathcal K}

\newcommand{\dd}{\,\mathrm d}

\newcommand{\Var}{\operatorname{Var}}
\newcommand{\Cov}{\operatorname{Cov}}

\newcommand{\braket}[1]{\left\langle #1\right\rangle}

\title{A quantitative replica-symmetric bound for Sherrington--Kirkpatrick model in the entire de Almeida--Thouless region}
\author{
  Seiichiro Kusuoka\thanks{Graduate School of Science, Kyoto University. E-mail: \texttt{kusuoka@math.kyoto-u.ac.jp}}
  \and
  Shuta Nakajima\thanks{Department of Mathematics, Keio University. E-mail: \texttt{njima@keio.jp}}
}
\date{}

\begin{document}
\maketitle

\begin{abstract}
We consider the Sherrington--Kirkpatrick model with inverse temperature $\beta>0$ and deterministic external
field $h>0$.  Let $q$ be  the replica-symmetric fixed point defined by 
$ q=\mathbb{E}\tanh^2\bigl(h+\beta\sqrt q\,Z\bigr)$, where $Z$ is a standard Gaussian random variable.  
We prove that, uniformly on compact subsets of the strict de Almeida--Thouless region, defined by $\beta^2\mathbb{E}\sech^4(h+\beta\sqrt q\,Z)<1$, the overlap satisfies the concentration estimate:
$$
 \mathbb{E}\braket{(R_{12}-q)^2}=O(N^{-1}).
$$
  As a consequence, we
obtain an $O(N^{-1})$ error bound for the replica-symmetric free energy   and
identify the finite-volume replicon susceptibility. Moreover, we prove the central limit theorem for the overlap in this region. 

The present paper provides an alternative proof of the replica-symmetric free energy formula in the de Almeida--Thouless region, which was recently established by Lopatto [arXiv:2604.11921]. An advantage of our approach is that  it establishes an explicit quantitative bound and yields the replica-symmetric free energy formula  as a consequence. 

Our proof is self-contained and does not use the identification of
the limiting free energy with the Parisi variational formula.

\end{abstract}

\section{Introduction}\label{sec:intro}

\subsection{Model and main results}

In this paper, we consider the Sherrington--Kirkpatrick model.
Let $N \in {\mathbb N}$ and $\Sigma_N\coloneqq\{-1,1\}^N$.  For $\beta>0$ and $h>0$, let
\[
 H_N(\sigma)
 \coloneqq
 \frac{\beta}{\sqrt{2N}}\sum_{i,j=1}^Ng_{ij}\sigma_i\sigma_j
 +h\sum_{i=1}^N\sigma_i , \quad \sigma \in \Sigma _N
\]
where $(g_{ij})_{1\le i,j\le N}$ are independent standard normal random variables. 

Let $q= q(\beta,h)$ be the unique solution of
\begin{equation}
 q=\mathbb{E}\tanh^2\bigl(h+\beta\sqrt q\,Z\bigr),
 \qquad Z\sim N(0,1),
 \label{eq:q}
\end{equation}
and write $q\coloneqq q(\beta,h)$.
The existence and uniqueness for all $\beta>0$ and  $h>0$ are obtained by a standard argument (see, for example, \cite[Proposition~1.3.8]{Talagrand}).
The continuity of $q(\beta ,h )$ in $(\beta , h)$ follows directly from  uniqueness.
Indeed, if $(\beta_n,h_n)\to(\beta,h)$, every convergent subsequence of $q(\beta_n,h_n)\in[0,1]$ has, by the dominated convergence theorem, the limit solving the fixed-point equation \eqref{eq:q} at $(\beta,h)$.
The uniqueness identifies that limit with $q(\beta,h)$, so
the entire sequence converges.
Set
\begin{equation}
 r\coloneqq\mathbb{E}\tanh^4\bigl(h+\beta\sqrt q\,Z\bigr),
 \label{eq:ra}
\end{equation}
and let
\begin{equation}
 \alpha\coloneqq\beta^2(1-2q+r)
 =\beta^2\mathbb{E}\sech^4\bigl(h+\beta\sqrt q\,Z\bigr).
 \label{eq:alpha}
\end{equation}
We work on a fixed compact set
\begin{equation}
 \cK\Subset\{(\beta,h):\beta>0,\ h>0,\ \alpha(\beta,h)<1\}.
 \label{eq:K}
\end{equation}
Then, there exists $\delta_{\cK}>0$ such that
\begin{equation}
 1-\alpha\ge\delta_{\cK}
 \qquad\text{on }\cK.
 \label{eq:delta}
\end{equation}
Since the maps $q$,$r$ and $\alpha$ are continuous on $\cK$, the compactness and $h>0$ imply
\begin{equation}
 0<q_{\cK}\coloneqq\inf_{(\beta,h)\in\cK}q
 \le \sup_{(\beta,h)\in\cK}q<1.
 \label{eq:qcompact}
\end{equation}
This inequality follows from the facts that $q=0$ would contradict \eqref{eq:q} with $h>0$, and that $q<1$ follows from the inequality $\tanh^2(h+\beta\sqrt q\,Z)<1$.

For $0\le s\le 1$, define the smart-path Hamiltonian $H_{N,s}$ by 
\begin{equation}
 H_{N,s}(\sigma)
 \coloneqq
 \frac{\beta\sqrt s}{\sqrt{2N}}
 \sum_{i,j=1}^Ng_{ij}\sigma_i\sigma_j
 +\sum_{i=1}^N
 \bigl(h+\beta\sqrt{(1-s)q}\,z_i\bigr)\sigma_i,
 \label{eq:path}
\end{equation}
where $(z_i)$ is a family of independent standard Gaussian random variables. Note that $H_{N,1}=H_N.$
We write $\braket{\cdot}_s$ for the Gibbs expectation, i.e. 
\[
\braket{F}_s \coloneqq \frac{1}{Z_{N, s}}\sum _{\sigma \in \Sigma _N} F(\sigma ) e^{H_{N, s}(\sigma)}
\]
for a function $F$ on $\Sigma _N$, where $Z_{N,s} \coloneqq \sum _{\sigma \in \Sigma _N} e^{H_{N,s}(\sigma )}$.
For replicas $\sigma^1,\sigma^2,\ldots$ which are sampled independently from the Gibbs measure for fixed disorder, put
\[
 R_{ab}\coloneqq\frac1N\sum_{i=1}^N\sigma_i^a\sigma_i^b.
\]
Set
\[
 \nu_s[f]\coloneqq\mathbb{E}\braket{f}_s,\qquad
 Q_{ab}\coloneqq R_{ab}-q.
\]
We also write
\begin{equation}
 \phi_N(s)\coloneqq\frac1N\mathbb{E}\log Z_{N,s}
 \label{eq:pathfreeenergy}
\end{equation}
and define
\begin{equation}
 A_s\coloneqq\nu_s[Q_{12}^2],\qquad
 B_s\coloneqq\nu_s[Q_{12}Q_{13}],\qquad
 C_s\coloneqq\nu_s[Q_{12}Q_{34}].
 \label{eq:ABC}
\end{equation}

Now we state the main theorem of this paper.

\begin{theorem}
\label{thm:main}
Under the setting above, the overlap concentration holds in the sense that there exists a constant $C_{\cK}<\infty$ such that, for every $N\ge1$,
\begin{equation}\label{eq:main1}
 \sup_{\substack{(\beta,h)\in\cK\\0\le s\le1}}
 N\,\nu_s[Q_{12}^2]\le C_{\cK}.
\end{equation}
Consequently, the expected free energy 
\[
 \phi_N(\beta,h)\coloneqq \phi_N(1) = 
 \frac1N\mathbb{E}\log\sum_{\sigma\in\Sigma_N}e^{H_N(\sigma)}
\]
and the replica-symmetric free energy
\[
 \phi^{\mathrm{RS}}(\beta,h)
 \coloneqq
 \log2+\mathbb{E}\log\cosh(h+\beta\sqrt q\,Z)
 +\frac{\beta^2}{4}(1-q)^2
\]
satisfy the quantitative estimate uniformly on $\cK$:
\begin{equation}\label{eq:main2}
 0\le\phi^{\mathrm{RS}}(\beta,h)-\phi_N(\beta,h)
 \le\frac{C_{\cK}}N.
\end{equation}
Moreover, it holds that
\begin{equation}\label{eq:main3}
 N\bigl(A_s-2B_s+C_s\bigr)
 =
 \frac{\alpha}{\beta^2(1-s\alpha)}+o_{\cK}(1).
\end{equation}
uniformly for $(\beta,h,s)\in\cK\times[0,1]$.
Here and below, the notation ``$o_{\cK}(1)\to0$" as $N\to\infty$ means the uniform convergence over the compact set $\mathcal K$.
\end{theorem}
It is worth noting that the estimate above was previously known in a sufficiently high-temperature regime from Talagrand's cavity analysis and from an independent unpublished interpolation argument of Lata{\l}{}a; see \cite{Talagrand}.  We recently refined Lata{\l}{}a's argument and enlarged the range of parameters \cite{KusuokaNakajima1}. The main contribution of the present work is to extend this estimate to the entire strict de Almeida--Thouless region. This is expected to be the maximal open region in which replica-symmetric overlap fluctuations remain on the $N^{-1/2}$ scale (see Theorem~\ref{thm:overlap-clt-intro} and Remark~\ref{remark: sigma diverge}).

By taking the limit as $N\rightarrow \infty$ in \eqref{eq:main2},\footnote{For points on the boundary of the de Almeida--Thouless region, the result follows by continuity from the interior. Indeed, the finite-volume free energies are uniformly Lipschitz in the model parameters, and hence their thermodynamic limit is continuous.} we recover replica symmetry
\[
\lim _{N\rightarrow \infty } \phi _N(\beta ,h) = \phi ^{\mathrm{RS}}(\beta ,h)
\]
throughout the entire de Almeida--Thouless region. Establishing replica symmetry in this region had remained an open problem for many years. Very recently, Lopatto \cite{Lopatto} resolved this problem by showing that the Parisi formula agrees with the replica-symmetric formula in this region, building on Talagrand's result that the limiting free energy is given by the Parisi formula \cite{TalagrandParisi}. Our main results provide a quantitative counterpart to this qualitative identification.

 As a consequence of the result above together with a standard cavity argument, we prove the central limit theorem for the overlap in the full strict de Almeida--Thouless region, which was previously known only in the cases where $\beta<1$ and $h=0$, and where $\beta$ is sufficiently small.
 
\begin{theorem}
\label{thm:overlap-clt-intro}
For every fixed $(\beta,h)\in\cK$, with respect to $\nu_1$, 
\[
 \sqrt N\,(R_{12}-q)
 \xrightarrow[N\to\infty]{\mathrm d}
 \mathcal N(0,\sigma^2),
\]
where
\[
 \sigma^2
 :=
 \frac{3(1-2q+r)}{1-\alpha}
 -\frac{2\kappa}{1-\beta^2\kappa}
 -\frac{\zeta}{(1-\beta^2\kappa)^2}>0,\]
\[ 
 \kappa\coloneqq1-4q+3r,
 \quad
 \zeta\coloneqq2q+q^2-3r.
\]
\end{theorem}

Note that when $(\beta,h)$ approaches the AT line $\{(\beta,h)\mid~\alpha(\beta,h)=1\}$, this $\sigma^2$ diverges (see  Remark~\ref{remark: sigma diverge}).  Hence, we do not expect the same claim  to hold  outside the regime.

In the following subsection, we briefly review the history of the Sherrington--Kirkpatrick model and the relevant previous work.

\subsection{A brief history and previous research} 
Spin glasses were introduced to describe magnetic materials in which the interactions between microscopic spins are both disordered and frustrated. Experiments on dilute magnetic alloys in the 1970s showed  behavior very different from that of ordinary ferromagnets: a sharp freezing transition, very slow relaxation, and a strong dependence on the history of the sample. These results suggested that the low-temperature phase could not be described by a single equilibrium state of the usual kind, and they led to probabilistic models with quenched disorder.

One of the first such models was proposed by Edwards and Anderson \cite{EdwardsAnderson1975}. In their model, spins sit on the vertices of a finite-dimensional lattice and interact through random nearest-neighbor couplings. This formulation is physically natural, but its mathematical analysis is extremely difficult. The local interactions compete with each other because of frustration, and the quenched disorder removes the spatial regularity that classical arguments rely on. As a result, even basic questions about the low-temperature phase remain difficult in finite dimensions.

A simplified model was proposed by Sherrington and Kirkpatrick \cite{SherringtonKirkpatrick1975}. They replaced the short-range geometry of the Edwards--Anderson model by a mean-field interaction, in which every spin interacts weakly with every other spin. The resulting model, now called the Sherrington--Kirkpatrick (SK) model, keeps both disorder and frustration, but is easier to study by analytic methods. Using the replica method, Sherrington and Kirkpatrick derived a formula for the limiting free energy under the assumption of replica symmetry. The replica method is not mathematically rigorous, because it starts from integer moments of the partition function and then uses an analytic continuation in the number of replicas. More importantly, the replica-symmetric formula gave a negative entropy at low temperature, which is physically inconsistent. It was soon understood that the replica-symmetric ansatz could not describe the whole phase diagram.

A fundamental step was taken by de Almeida and Thouless \cite{deAlmeidaThouless1978}, who studied the local stability of the replica-symmetric solution. 
Their computation predicted that the replica-symmetric solution is  stable when
\begin{align}
    \label{eq: AT region}
\beta^2\mathbb E\operatorname{sech}^4\bigl(h+\beta\sqrt q\,Z\bigr)\leq 1,
\end{align}
and unstable when the reverse strict inequality holds. This boundary is now called the de Almeida--Thouless line, or AT line. Thus, the AT condition comes from the physical stability analysis of the replica-symmetric state.

The failure of replica symmetry below the AT line led Parisi to propose a much richer picture of the low-temperature phase (see \cite{Parisi1979,Parisi1980}). In this picture, the symmetry among replicas is broken in a hierarchical way, and the thermodynamic state is described by an order parameter that gives the distribution of overlaps between typical configurations. This led to the Parisi variational formula for the limiting free energy. The Parisi theory resolved the low-temperature problem of the replica-symmetric solution and gave a single framework that covers both the replica-symmetric and the replica-symmetry-breaking phases.

These ideas have been used far beyond condensed matter physics. In computer science, random constraint satisfaction problems such as random $K$-SAT and random graph coloring show threshold phenomena similar to replica symmetry breaking, and algorithms based on the cavity method, such as belief and survey propagation, are still among the best known methods for these problems (see \cite{MezardParisiZecchina2002,MezardMontanari2009}). In information theory and high-dimensional statistics, the free energy of suitable spin glass models characterizes the fundamental limits in error-correcting codes, compressed sensing, and inference in graphical models (see \cite{Nishimori2001,ZdeborovaKrzakala2016}).   In machine learning, the Hopfield model of associative memory is also a spin glass model in which the stored patterns play the role of quenched disorder (see \cite{Hopfield1982}), and spin glass methods have been used to study loss landscapes and training dynamics of neural networks. More generally, the notions of frustration, energy landscapes, and metastability, identified in spin glasses, are now standard tools in the study of complex systems, from protein folding to financial networks.

The rigorous mathematical study of the SK model has developed along  several different routes. An early landmark was the work of  Aizenman, Lebowitz, and Ruelle \cite{AizenmanLebowitzRuelle1987}, who proved Gaussian fluctuations of the free energy at high temperature and zero external field by using a cluster expansion. At zero external field, replica symmetry is now known to hold for all $\beta<1$, and the critical point $\beta=1$ separates the replica-symmetric and symmetry-breaking regimes. With an external field, the problem turned out to be more difficult. Talagrand  proved the replica-symmetric formula at sufficiently high temperature by the cavity method, and a simple interpolation argument due to Lata\l{}a gave another proof in the range $\beta \in [0,1/2)$ and $h\in \R$. A recent work of the authors extended Lata\l{}a's argument to a wider range including the region $\beta\leq 1$ and $h\in \R $ (see \cite{KusuokaNakajima1}).

In parallel, a different line of work addressed the limiting free energy at all temperatures.  Guerra \cite{Guerra}  introduced an interpolation argument that proved an appropriate upper bound, and then  Talagrand \cite{TalagrandParisi} proved the matching opposite bound.  Together, these results established the Parisi formula for the SK model. This turned the study of the phase diagram of the free energy into a rigorous variational problem. 

One question remained: whether the AT stability condition exactly characterizes the replica-symmetric phase of the SK model with a nonzero external field or not.  Several rigorous results proved replica symmetry in regions inside the predicted AT region, or replica symmetry breaking on the other side of the line, but reaching the AT boundary itself remained difficult and had been a major open problem in this field. Very recently,  Lopatto \cite{Lopatto} proved that replica symmetry holds throughout the AT region \eqref{eq: AT region} 
 using the variational characterization of the Parisi formula.  The present paper gives a quantitative counterpart of this picture  (see Theorem~\ref{thm:main} and the explanation after that).

 The fluctuations of the overlap have also been studied extensively. In zero external field, for $\beta<1$, Aizenman, Lebowitz and Ruelle proved in \cite{AizenmanLebowitzRuelle1987} a central limit theorem for the overlap by chaos expansion (see  also \cite{CometsNeveu} for a different proof via stochastic calculus). At the critical point $(\beta,h)=(1,0)$, Du and Huang very recently identified the critical $N^{-1/3}$ scale and determined the limiting quenched distribution of the rescaled overlap (see \cite{DuHuangCriticalOverlap}). For nonzero external field, Gaussian fluctuations of the centered overlaps on the $N^{-1/2}$ scale were established, in a sufficiently high-temperature region, independently by Talagrand and by Guerra and Toninelli (see \cite{TalagrandRSB,GuerraToninelliCLT}). So, our second main theorem extends this to a maximal regime where the claim holds.

\subsection{Outline of the proof: from coupled free energy to optimal overlap bound}
\label{subsec:proof-outline}

We briefly sketch the proof of the first claim in the first main result.
The second and third claims  are obtained from this using a standard interpolation argument (see Section~\ref{sec:conclusion}). 

Our proof is inspired by that of the Parisi formula in \cite{TalagrandParisi}. For $u \in \mathbb R$, define the coupled 
two-replica pressure by
\begin{equation}
 p^{(2)}_{N,s}(u )
  :=
 \frac1{2N}\mathbb{E}\log
 \sum_{\sigma^1,\sigma^2}
 \exp\left(
 H_{N,s}(\sigma^1)+H_{N,s}(\sigma^2)
 +\frac{u  N}{2}Q_{12}^2\right).
\label{eq:constrained-two-replica-pressure}
\end{equation}
The Guerra--Talagrand interpolation gives
\begin{equation}
p_{N,s}^{(2)}(u )
\le
\mathfrak P_s(u )+ o_\cK(1)
\label{eq:GT-constrained-bound-outline}
\end{equation}
uniformly in $(\beta,h)\in\mathcal K$, $s\in[0,1]$, and
$u\in[-1,1]$. Here $\mathfrak P_s(u )$ denotes the
corresponding one-step replica-symmetry-breaking formula. Using an argument similar to that of Lopatto's  work \cite{Lopatto},  one obtains the flatness condition that for $u>0$ small enough,
 \begin{equation}
\mathfrak P_s(u )
= \mathfrak P_s(0). %+ o(u ).
\label{eq:global-GT-gap-outline}
\end{equation}

Let \begin{equation*}
 D_N(s)\coloneqq \mathfrak P_s(0)-\phi_N(s).
\end{equation*} 
By computations using Gaussian integration by parts, we have for $u_0>0$ sufficiently small,  
\[
D_N'(s) =  \frac{\beta^2}{4} \nu_s [Q_{12}^2]  \leq \frac{\beta^2}{u_0} ({p_{N,s}^{(2)}(u_0)-\phi_N(s)}) \leq   \frac{\beta^2}{u_0} D_N(s) +  o_\cK(1),
\]
where we have used Jensen's inequality with $p_{N,s}^{(2)}(u_0)-\phi_N(s)= (2N)^{-1}\mathbb E\log \langle e^{u_0 N Q_{12}^2/2} \rangle_s$ for the first inequality, and \eqref{eq:GT-constrained-bound-outline} and \eqref{eq:global-GT-gap-outline} for the last inequality. Since $D_N(0) = 0$, by Gronwall's inequality, we have
$$D_N(s) = o_\cK(1),\quad  p_{N,s}^{(2)}(u_0)-\phi_N(s)= \frac{1}{2N}\mathbb E\log \langle e^{u_0 N Q_{12}^2/2} \rangle_s = o_\cK(1).$$ 

For a fixed $\delta>0$, we have 
\begin{align*}
        \nu_s\left(
|Q_{12}|\ge\delta
\right)&\le \mathbb P( \langle e^{u_0 N Q_{12}^2/2} \rangle_s\ge e^{\delta^2 u_0 N/4}) + \mathbb E \Big[\mathbf{1}\{\langle e^{u_0 N Q_{12}^2/2} \rangle_s<e^{\delta^2 u_0 N/4}\}\langle \mathbf{1}\{|Q_{12}|\ge\delta\}\rangle_s\Big].
\end{align*}
For the first term,  using a standard Gaussian Lipschitz concentration inequality for $\log \langle e^{u_0 N Q_{12}^2/2} \rangle_s$,  one gets 
\begin{align*}
\mathbb P( \langle e^{u_0 N Q_{12}^2/2} \rangle_s\ge e^{\delta^2 u_0 N/4})\leq C_{\mathcal K,\delta}
\exp(-c_{\mathcal K,\delta}N).
\end{align*}
For the  second term, by Markov's inequality, we have
\begin{align*}
    &\mathbb E \Big[\mathbf{1}\{\langle e^{u_0 N Q_{12}^2/2} \rangle_s<e^{\delta^2 u_0 N/4}\}\langle \mathbf{1}\{|Q_{12}|\ge\delta\}\rangle_s\Big]\\
    &\leq \mathbb E \Big[\mathbf{1}\{\langle e^{u_0 N Q_{12}^2/2} \rangle_s<e^{\delta^2 u_0 N/4}\}\langle e^{u_0 N Q_{12}^2/2}\rangle_s e^{-u_0 \delta^2 N/2}\Big]\\
&\le
C_{\mathcal K,\delta}
\exp(-c_{\mathcal K,\delta}N).
\end{align*}
Therefore, we obtain a weak concentration for the overlap: 
\begin{equation}
\begin{split}
    \nu_s\left(
|Q_{12}|\ge\delta
\right)&\leq 
2 C_{\mathcal K,\delta}
\exp(-c_{\mathcal K,\delta}N).
\label{eq:fixed-deviation-overlap-tail-outline}
\end{split}
\end{equation}

To obtain the $N^{-1}$ bound, we use
the cavity bootstrap developed by  Talagrand in his high-temperature analysis (see \cite[Chapter~1]{Talagrand}).
Estimating the influence of adding one spin, we reach
\begin{equation}
\nu_s[Q_{12}^2] \leq C_\cK 
 \left(
\nu_s|Q_{12}|^3+N^{-1}
\right).
\label{eq:cavity-remainder-outline}
\end{equation}
The weak concentration estimate \eqref{eq:fixed-deviation-overlap-tail-outline} yields that for any fixed
$\delta >0$,
\begin{align}
\nu_s|Q_{12}|^3
&\le
\delta\nu_s[Q_{12}^2]
+
8\nu_s\left(|Q_{12}|\ge\delta\right) \le
\delta \nu_s[Q_{12}^2]+
C_{\mathcal K,\delta}e^{-c_{\mathcal K,\delta}N}.
\label{eq:cubic-bootstrap-outline}
\end{align}
Consequently, we have
\[
\nu_s[Q_{12}^2] \leq \delta C_\cK \nu_s[Q_{12}^2] +  C_\cK N^{-1} + C_\cK C_{\cK,\delta}e^{-c_{\mathcal K,\delta}N}.
\]
Choosing $\delta=1/(2C_\cK)>0$, we obtain
\begin{equation}
\nu_s[Q_{12}^2]
\le
\frac{C'_{\mathcal K}}{N},
\label{eq:optimal-overlap-bound-outline}
\end{equation}
with some $C'_\cK$.
This is the sketch of the proof of the quantitative bound for overlaps.

\section*{Formalization in Lean 4}

The main results of this paper have been formalized in Lean 4. The formal artifact is available in the GitHub repository
\[
\texttt{njimaMath/research\_public/RSAT}.
\]

The file \texttt{Main.lean} provides the public interface between the mathematical statements in this paper and their formal counterparts.

The smart-path Hamiltonian \eqref{eq:path} is represented by \texttt{H\_N\_s}, and the theorem \texttt{H\_N\_s\_eq\_smartPath} verifies that this concrete Hamiltonian agrees with the abstract smart-path implementation used by the formal proof.

Theorem~\ref{thm:main} corresponds to the Lean theorem
\texttt{strictAT\_main}. Its conclusion is represented by the structure
\texttt{StrictATClaim}, whose three fields formalize respectively:
\begin{itemize}
    \item the uniform estimate $N\nu_s[Q_{12}^2]\le C_{\mathcal K}$;
    \item the signed $O(N^{-1})$ estimate for the replica-symmetric free energy correction;
    \item the asymptotic formula for the finite-volume replicon susceptibility.
\end{itemize}

Theorem~\ref{thm:overlap-clt-intro} corresponds to
\texttt{strictAT\_overlapCLT\_weak}. This theorem gives weak convergence of
$\sqrt{N}(R_{12}-q)$ to the centered Gaussian distribution with the variance stated in Theorem~\ref{thm:overlap-clt-intro}.

The formalization uses Lean 4 together with Mathlib. The repository contains the complete dependency structure, build instructions, and the source files for the cavity estimates, overlap concentration, Guerra--Talagrand argument, Gaussian calculations, smart-path interpolation, and the central limit theorem. In particular, the formal statements use a concrete countable product Gaussian probability space, and \texttt{Main.lean} supplies the bridge between this implementation and the notation used in the present paper.

\subsection{Organization of the present paper}

The rest of the paper is devoted to the proof of Theorem~\ref{thm:main} and Theorem~\ref{thm:overlap-clt-intro}.
In Section~\ref{sec:coercivity}, we establish coercivity estimates for the replica-symmetric solution.   
The argument in this part is inspired by Talagrand's approach in \cite{TalagrandParisi}, and involves several  explicit technical  calculations. 
Section~\ref{sec:concentration} gives estimates on overlap concentration. This part is similar to the argument in \cite{KusuokaNakajima1}. In Section~\ref{app:cavity-identities}, we prove  Proposition~\ref{prop:cavity}, which provides the overlap estimates used in Section~\ref{sec:concentration}. In Section~\ref{sec:conclusion} we prove the main theorem (Theorem~\ref{thm:main}) by applying the results in Sections~\ref{sec:coercivity}~and~\ref{sec:concentration}. In Section~\ref{sec:overlap-clt}, we prove the central limit theorem for the overlap (Theorem~\ref{thm:overlap-clt-intro}). 

In the appendices, we prove several technical results used in Section~\ref{sec:coercivity}. 
 Section~\ref{app:specialGT} is about the finite-volume Guerra--Talagrand upper bound, which is applied in Section~\ref{subsec:GT}.
In Section~\ref{sec:Price} we state and prove Price's theorem, which is also used in Section~\ref{subsec:GT}. 

We also note that we provide a formalization of the proof in Lean 4 \cite{brabra}.

We use \(C_{\mathcal K}\) and \(c_{\mathcal K}\) to denote, respectively, large and small positive constants depending only on the compact set \(\mathcal K\). Their values may vary from line to line.

\section{Replica-symmetric coercivity}\label{sec:coercivity}

\subsection{Estimates on the replica-symmetric path}\label{subsec:ATsign}

Fix $(\beta,h,s)\in\cK\times[0,1]$.
For $t\ge0$, let $\mathsf H_t$ be the heat semigroup, i.e. 
\[
 (\mathsf H_t\varphi)(x)\coloneqq\E\varphi(x+\sqrt t\,Z),
\]
where $Z$ is a standard normal random variable.
Then, it holds that
\begin{equation}\label{eq:genH}
\partial _t (\mathsf H_t\varphi)(x) = \frac{1}{2} \mathsf H_t (\partial ^2 \varphi)(x) .
\end{equation}
We introduce the tilted heat semigroup
\begin{equation}
 (\mathsf T_t\varphi)(x)
 \coloneqq e^{-t/2}
 \frac{\bigl(\mathsf H_t
   (\varphi \cdot \cosh )\bigr)(x)}{\cosh(x)}.
 \label{eq:tiltedsemigroup}
\end{equation}
Differentiation under the Gaussian integral shows that its generator is
$\frac12\partial_{xx}+\tanh(x)\partial_x$.  For each $u\in[0,1]$,
let $X_{s,u}$ denote a random variable whose law is specified by
\begin{equation}
 \E\varphi(X_{s,u})
 \coloneqq
 \begin{cases}
  \bigl(\mathsf H_{\beta^2(1-s)q+s\beta^2u}\varphi\bigr)(h),
      &0\le u\le q,\\[1mm]
  \bigl(\mathsf H_{\beta^2q}
      (\mathsf T_{s\beta^2(u-q)}\varphi)\bigr)(h),
      &q\le u\le1.
 \end{cases}
 \label{eq:localfieldlaw}
\end{equation}
In particular, it holds that
\begin{equation}\label{eq:Xsq}
X_{s,q}\stackrel{\mathrm d}=h+\beta\sqrt q\,Z,
\end{equation}
where ``$\stackrel{\mathrm d}=$" means the equality in law.

\begin{lemma}
\label{lem:uppercomparison}
If $s\beta^2(1-q)>1$ and $q\le u\le1$, then
\begin{equation}
 \beta^2\mathbb{E}\sech^4(X_{s,u})\le\alpha.
 \label{eq:uppercomparison}
\end{equation}
\end{lemma}

\begin{proof}
We first show that 
\begin{equation}
 h<\sigma^2:=\beta^2    q .
 \label{eq:h-less-sigma}
\end{equation}
Indeed, suppose that $h\ge\sigma^2$. A direct calculation shows that 
\[
 1-q
 =
 \mathbb{E}\sech^2(h+\sigma Z)
 \le
 \mathbb{E}\sech^2(\sigma^2+\sigma Z).
\]
Let $\varepsilon$ be a symmetric Rademacher variable, independent of $Z$,
and put $Y=\sigma^2\varepsilon+\sigma Z$. Then $\varepsilon^2=1$ and 
$\E[\varepsilon\mid Y]=\tanh Y$ due to $\mathbb P(\varepsilon=\pm 1,\,Y\in {\mathrm d}y) = \frac{1}{2\sqrt{2\pi \sigma^2}}e^{-(y - (\pm \sigma^2) )^2/2\sigma^2} {\mathrm d}y$, and hence
\[
 \mathbb{E}\sech^2(\sigma^2+\sigma Z)
 =
 \E\operatorname{Var}(\varepsilon\mid Y)
 \le
\E\Big(\varepsilon-\frac{1}{1+\sigma^2}Y\Big)^2
 =
 \frac{1}{1+\sigma^2}.
\]
Therefore
\[
 (1-q)(1+\beta^2q)\le1.
\]
Reorganizing this with $q>0$ implies $\beta^2(1-q)\le1$. This contradicts the assumption 
$s\beta^2(1-q)>1$ because $s\le1$.
Thus, we obtain \eqref{eq:h-less-sigma}.

Fix $q\le u\le1$ and set
\[
 t\coloneqq s\beta^2(u-q),
 \qquad
 X\coloneqq h+\sigma Z.
\]
From \eqref{eq:localfieldlaw} it follows that
\[
 \mathbb{E}\sech^4(X_{s,u})
 =
 \E\bigl[(\mathsf T_t\sech^4)(X)\bigr].
\]
Hence, for \eqref{eq:uppercomparison} it suffices to prove
\[
 \E\bigl[(\mathsf T_t\sech^4)(X)\bigr]
 \le
 \mathbb{E}\sech^4(X).
\]

Now put
\[
 g_t\coloneqq\mathsf T_t \sech^4 \qquad \mbox{and}\quad
 \qquad
 M(t)\coloneqq\E g_t(X).
\]
Since
\[
 g_t(x)
 =
 e^{-t/2}\frac{\mathsf H_t(\sech^3)(x)}{\cosh x},
\]
$g_t$ is an even and nonincreasing function on $[0,\infty)$.
In particular, it holds that
\[
 g_t'(x)\le0,
 \qquad x>0.
\]
By \eqref{eq:genH} and \eqref{eq:tiltedsemigroup} we have
\begin{equation}\label{eq:partialT_t}
 \partial_t(\mathsf T_t\varphi)(x)
 =\frac{1}{2\cosh^2x}\partial_x\left(
   \cosh^2x\,\partial_x(\mathsf T_t\varphi)(x)
  \right).
\end{equation}
Let $\phi(x)=(2\pi)^{-1/2}e^{-x^2/2}$ be the standard normal density.
Since $X\sim N(h,\sigma^2)$, its density is
$\sigma^{-1}\phi((x-h)/\sigma)$.
From \eqref{eq:partialT_t} and integration by parts
in divergence form we have
\[
 \begin{aligned}
 M'(t)
 &=\frac12\int_{\mathbb R}
   \partial_x\bigl(\cosh^2x\,g_t'(x)\bigr)
   \frac{\phi((x-h)/\sigma)}{\sigma\cosh^2x}\,\dd x\\
 &=-\frac12\int_{\mathbb R}\cosh^2x\,g_t'(x)
   \partial_x\left(
     \frac{\phi((x-h)/\sigma)}{\sigma\cosh^2x}
   \right)\dd x\\
 &=\E\left[
   \left(\tanh X+\frac{X-h}{2\sigma^2}\right)g_t'(X)
 \right].
 \end{aligned}
\]
Pairing $x$ and $-x$
gives
\begin{equation}\label{eq:M'}
 M'(t)=\int_0^\infty g_t'(x)K(x)\,\dd x,
\end{equation}
where
\[
 \begin{aligned}
 K(x)
 &\coloneqq
 \frac1\sigma\phi\left(\frac{x-h}{\sigma}\right)
 \left(\tanh x+\frac{x-h}{2\sigma^2}\right)+
 \frac1\sigma\phi\left(\frac{-x-h}{\sigma}\right)
 \left(\tanh x+\frac{x+h}{2\sigma^2}\right).
 \end{aligned}
\]
Since
\[
 \frac{\phi((-x-h)/\sigma)}{\phi((x-h)/\sigma)}
 =e^{-2hx/\sigma^2},
\]
we obtain
\[
 \frac{\sigma K(x)}{
  \phi((x-h)/\sigma)+\phi((-x-h)/\sigma)}
 =
 \tanh x+\frac{x}{2\sigma^2}
 -\frac{h}{2\sigma^2}\tanh\left(\frac{hx}{\sigma^2}\right).
\]
Since $0<h/\sigma^2<1$ due to \eqref{eq:h-less-sigma},
\(
 \tanh\left(\frac{hx}{\sigma^2}\right)\le\tanh x .
\)
This inequality  implies that for $x>0$
\[
 \frac{\sigma K(x)}{
  \phi((x-h)/\sigma)+\phi((-x-h)/\sigma)}
 \ge
 \left(1-\frac{h}{2\sigma^2}\right)\tanh x
 +\frac{x}{2\sigma^2}
 >0.
\]
Thus $K(x)>0$, while $g_t'(x)\le0$, and therefore \eqref{eq:M'} yields
\[
 M'(t)\le0.
\]
In particular, it holds that $M(t)\le M(0)$. From this inequality we obtain
\[
 \beta^2\mathbb{E}\sech^4(X_{s,u})
 \le
 \beta^2\mathbb{E}\sech^4(h+\beta\sqrt q\,Z)
 =\alpha.
\]
\end{proof}
For $0\le u\le1$ and $x\in\mathbb R$, define
\begin{equation}
 \Psi(u,x)
 \coloneqq
 \E\log\cosh\left(
 x+\beta\sqrt{s(q-u)_+}\,Z
 \right)
 +\frac{s\beta^2}{2}\bigl(1-\max\{u,q\}\bigr),
 \qquad Z\sim N(0,1),
 \label{eq:Psi-explicit}
\end{equation}
where $(a)_+\coloneqq\max\{a,0\}$.  Equivalently, we can define $\Psi (u,x)$ by
\[
 \Psi(u,x)
 =
 \begin{cases}
  (\mathsf H_{s\beta^2(q-u)}\log\cosh)(x)
  +\dfrac{s\beta^2}{2}(1-q),
  &0\le u\le q,\\[2mm]
  \log\cosh(x)+\dfrac{s\beta^2}{2}(1-u),
  &q\le u\le1.
 \end{cases}
\]
Differentiating \eqref{eq:Psi-explicit} in $x$ gives
\begin{equation}
 \partial_x\Psi(u,x)
 =
 \mathbb{E}\tanh\left(
 x+\beta\sqrt{s(q-u)_+}\,Z
 \right)
 =
 \begin{cases}
  (\mathsf H_{s\beta^2(q-u)}\tanh)(x),
      &0\le u\le q,\\
  \tanh(x),&q\le u\le1.
 \end{cases}
 \label{eq:Psi-x}
\end{equation}

The replica-symmetric free energy for $H_{N,s}$ along the path is defined by
\begin{align}
 P_s^*& \coloneqq\log2+
 \E\Psi(0,h+\beta\sqrt{(1-s)q}\,Z)
 -\frac{s\beta^2}{2}\int_q^1u\,\dd u
 \label{eq:scalartrial}\\
 &= \log2+\E\log\cosh(h+\beta\sqrt q\,Z)
 +\frac{s\beta^2}{4}(1-q)^2,
 \label{eq:RSpathvalue}
\end{align}
where we have used the semigroup property and
$\beta^2(1-s)q+s\beta^2q=\beta^2q$.

\begin{proposition}
\label{prop:pathRS}
The function $g_s$ defined by
\begin{equation}
 g_s(u)\coloneqq\E\bigl[(\partial_x\Psi(u,X_{s,u}))^2\bigr]
 \label{eq:gs-def}
\end{equation}
satisfies
\begin{equation}
 g_s(u)-u
 \begin{cases}
 >0,&0\le u<q,\\
 =0,&u=q,\\
 <0,&q<u\le1.
 \end{cases}
 \label{eq:ATsign}
\end{equation}
Moreover, there are constants $c_{\cK},\varepsilon_{\cK}>0$ such that
\begin{equation}
 |g_s(u)-u|\ge c_{\cK}|u-q|
 \qquad\text{whenever }|u-q|\le\varepsilon_{\cK}.
 \label{eq:linearATsign}
\end{equation}
\end{proposition}

\begin{proof}
Since \eqref{eq:Psi-x} and \eqref{eq:Xsq} imply
\[
g_s(q) = \E\bigl[ \tanh ^2 (X_{s,q}) \bigr] =q ,
\]
\eqref{eq:linearATsign} holds for $u=q$.

We next consider the case $u<q$. 
Since \eqref{eq:Psi-x} gives
\[
g_s(u) = \E\left[ \bigl(\mathsf H_{s\beta^2(q-u)} (\tanh )\bigr)(X_{s,u}) ^2 \right] ,
\]
by noting that \eqref{eq:genH} and \eqref{eq:localfieldlaw} imply
\[
\partial _u \E\varphi(X_{s,u}) = \frac{s\beta ^2}{2} \bigl(\mathsf H_{\beta^2(1-s)q+s\beta^2u} (\partial ^2 \varphi ) \bigr) (h) = \frac{s\beta ^2}{2} \E\left[ (\partial ^2 \varphi ) (X_{s,u}) \right] ,
\]
we have
\begin{align*}
g_s'(u) &= -s\beta^2 \left. \partial _v \E\left[ \bigl(\mathsf H_v (\tanh )\bigr)(X_{s,u}) ^2 \right] \right| _{v=s\beta^2(q-u)} + \left. \partial _v \E\left[ \bigl(\mathsf H_{s\beta^2(q-u)} (\tanh )\bigr)(X_{s,v}) ^2 \right] \right| _{v=u} \\
&= - s\beta^2 \E\left[ \bigl(\mathsf H_{s\beta^2(q-u)} (\tanh )\bigr)(X_{s,u}) \bigl(\mathsf H_{s\beta^2(q-u)} ( \partial ^2 \tanh )\bigr)(X_{s,u}) \right] \\
&\quad + \frac{s\beta ^2}{2} \E\left[ \left. \left( \partial _x ^2 \bigl(\mathsf H_{s\beta^2(q-u)} (\tanh ) (x)\bigr) ^2 \right) \right| _{x=X_{s,u}} \right] \\
&= s\beta ^2 \E\left[ \bigl(\mathsf H_{s\beta^2(q-u)} (\partial \tanh )\bigr) (X_{s,u}) ^2 \right] .
\end{align*}
Thus, we obtain
\[
 g_s'(u)
 =s\beta^2\E\left[
   \left\{
    \bigl(\mathsf H_{s\beta^2(q-u)}
      ( \sech^2 )\bigr)(X_{s,u})
   \right\}^2
 \right].
\]
Jensen's inequality, the semigroup property and \eqref{eq:localfieldlaw} imply
\[
 g_s'(u)
 \le s\beta^2\E\left[
   \bigl(\mathsf H_{s\beta^2(q-u)}
     ( \sech^4 )\bigr)(X_{s,u})
 \right]
 =s\beta^2\mathbb{E}\sech^4(h+\beta\sqrt q\,Z)
 =s\alpha.
\]
Since $g_s(q)=q$, the integration of both sides implies
\[
q - g_s(u) \leq s \alpha (q-u) .
\]
Hence, it holds that
\begin{equation}
 g_s(u)-u\ge(1-s\alpha)(q-u)
 \qquad(0\le u<q).
 \label{eq:leftstrict}
\end{equation}
This yields \eqref{eq:ATsign} for $0\le u<q$.

We consider the case $u\ge q$.
In this case, \eqref{eq:Psi-x} gives $g_s(u)=\mathbb{E}\tanh^2(X_{s,u})$.
On the other hand, \eqref{eq:genH}, \eqref{eq:tiltedsemigroup} and \eqref{eq:localfieldlaw} imply
\begin{align*}
\partial _u \E\varphi(X_{s,u}) &= \partial _u \mathsf H_{\beta ^2 q} \left( e^{-\frac{s\beta ^2}{2}(u-q)} \frac{\mathsf H_{s\beta ^2 (u-q)}(\varphi \cdot \cosh )}{\cosh} \right) (h) \\
&= -\frac{s\beta ^2}{2}\E\varphi(X_{s,u}) + \frac{s\beta ^2}{2} \mathsf H_{\beta ^2 q} \left( e^{-\frac{s\beta ^2}{2}(u-q)} \frac{\mathsf H_{s\beta ^2 (u-q)}( \partial ^2 (\varphi \cdot \cosh) )}{\cosh} \right) (h) \\
&= \frac{s\beta ^2}{2} \E (\partial ^2 \varphi )(X_{s,u}) + s\beta ^2 \mathsf H_{\beta ^2 q} \left( e^{-\frac{s\beta ^2}{2}(u-q)} \frac{\mathsf H_{s\beta ^2 (u-q)}( ( \partial \varphi ) \cdot \sinh )}{\cosh} \right) (h).
\end{align*}
Hence, it holds that
\begin{align*}
g_s'(u) &= s\beta ^2 \E \sech ^4 (X_{s,u}) -2 s\beta ^2 \E \tanh ^2(X_{s,u}) \sech ^2 (X_{s,u})\\
&\quad + 2 s\beta ^2 \mathsf H_{\beta ^2 q} \left( e^{-\frac{s\beta ^2}{2}(u-q)} \frac{\mathsf H_{s\beta ^2 (u-q)}( \sech \cdot \tanh ^2 )}{\cosh} \right) (h).
\end{align*}
Thus, by recalling \eqref{eq:tiltedsemigroup} and \eqref{eq:localfieldlaw}, we have
\begin{equation}
 g_s'(u)=s\beta^2\mathbb{E}\sech^4(X_{s,u}),
 \qquad g_s'(q)=s\alpha<1.
 \label{eq:gs-upper-derivative}
\end{equation}

If $s\beta^2(1-q)>1$, Lemma~\ref{lem:uppercomparison} gives
$g_s'(u)\le s\alpha <1$, so $g_s(u)<u$ for $u>q$.
We consider the case $s\beta^2(1-q)\le1$.
Since
$\sech^4(x)\le\sech^2(x)=1-\tanh^2(x)$ for every $x\in\mathbb R$,
by \eqref{eq:gs-upper-derivative} we have
\begin{equation}\label{eq:gderivative}
 g_s'(u)\le s\beta^2(1-g_s(u)).
\end{equation}
Then, \eqref{eq:gderivative} implies that for $u>q$
\[
\partial _u (g_s(u) -u) \leq -s\beta ^2 (g_s(u)-u) + s\beta ^2 (1-u) -1 \leq -s\beta ^2 (g_s(u)-u) .
\]
Hence, by noting that \eqref{eq:gs-upper-derivative} implies $g_s(u)-u <0$ for $u>q$ sufficiently close to $q$ and by applying Gr\"onwall's inequality we have
\[
g_s(u) - u < 0 \qquad (u>q) .
\]
This yields \eqref{eq:ATsign} for $u>q$.

To prove \eqref{eq:linearATsign}, note first that \eqref{eq:leftstrict} and
\eqref{eq:delta} give
\[
 g_s(u)-u\ge\delta_{\cK}(q-u),
 \qquad 0\le u<q.
\]
Moreover, the expression for $g_s'(u)$ in \eqref{eq:gs-upper-derivative}
is jointly continuous in $(\beta,h,s,u)$ on the set $q\le u\le1$.
Since $g_s'(q)=s\alpha\le1-\delta_{\cK}$, compactness gives
$\varepsilon_{\cK}>0$ such that $\text{whenever }q\le u\le q+\varepsilon_{\cK}$,
\[
 g_s'(u)\le1-\frac{\delta_{\cK}}2.
\]
Thus, for $q<u\le q+\varepsilon_{\cK}$,
\[
 g_s(u)-u
 =\int_q^u\bigl(g_s'(v)-1\bigr)\,\dd v
 \le-\frac{\delta_{\cK}}2(u-q).
\]
Putting things together, we have 
\eqref{eq:linearATsign} with $c_{\cK}=\delta_{\cK}/2$ after decreasing
$\varepsilon_{\cK}$ if necessary.
\end{proof}

\subsection{Two-replica Guerra--Talagrand coercivity}\label{subsec:GT}

For $\lambda\in\mathbb R$, define a function $f_\lambda$ on ${\mathbb R}^2$ by
\begin{equation}
 f_\lambda(x_1,x_2)
 \coloneqq\log\left(\frac14\sum_{\varepsilon_1,\varepsilon_2=\pm1}
 \exp(\varepsilon_1x_1+\varepsilon_2x_2
       +\lambda\varepsilon_1\varepsilon_2)\right).
 \label{eq:GTterminal}
\end{equation}
For $(\beta,h,s)\in\cK\times[0,1]$, define a function $\xi_s$ on ${\mathbb R}$ by
\begin{equation}
 \xi_s(r)\coloneqq\beta^2(1-s)qr+\frac{s\beta^2}{2}r^2.
\label{eq:xi}
\end{equation}
Given $v\in[-1,1]$, put $\iota_v\coloneqq1$ for $v\ge0$ and
$\iota_v\coloneqq-1$ for $v<0$.  Let $\{ a_j \}_{j=0}^{M}$ be the increasing list of the distinct elements of $\{0,q,|v|,1\}$, i.e.
\[
\{ a_j \}_{j=0}^{M}= \{0,q,|v|,1\} \quad \mbox{and} \quad a_j< a_{j+1} \ (0\le j \le M-1).
\]
Note that $M=2$ if $|v| \in \{0,q,1\}$, and $M=3$ otherwise.
Define
\begin{equation}
 Q_j^v\coloneqq
 \begin{pmatrix}
 a_j&\iota_v(a_j\wedge |v|)\\
 \iota_v(a_j\wedge |v|)&a_j
 \end{pmatrix},
 \qquad 0\le j\le M.
\end{equation}
For $0\le j\le M$, set
\begin{equation}
 m_j\coloneqq
 \begin{cases}
 0,&a_j\le q,\\[1mm]
 \frac12,&q<a_j\le |v|,\\[1mm]
 1,&a_j>\max\{q,|v|\}.
 \end{cases}
\end{equation}
We call $(a_j)_{j=0}^M$ the breakpoints and $(m_j)_{j=0}^{M}$ the levels.
For each $v$, the  breakpoints  and levels are given as follows.

\begin{center}
\begin{tabular}{|c||c|c|}
\hline
& the  breakpoints  $(a_j)_{j=0}^M$ & the levels $(m_j)_{j=0}^{M}$ \\[1mm]
\hline
$|v|=0$ & $0,q,1$ & $0,0,1$\\[1mm]
\hline
$0<|v|<q$ & $0,|v|,q,1$ & $0,0,0,1$ \\[1mm]
\hline
$|v|=q$ & $0,q,1$ & $0,0,1$ \\[1mm]
\hline 
$q<|v|<1$ & $0,q,|v|,1$ & $0,0,\frac12,1$ \\[1mm]
\hline
$|v|=1$ & $0,q,1$ & $0,0,\frac12$\\[1mm]
\hline
\end{tabular}
\end{center}
% \footnote{\textcolor{red}{SN: The table heading says the levels are $(m_j)_{j=0}^M$, but the rows list only $(m_j)_{j=1}^M$. For example, when $|v|=0$, the full level vector is $(0,0,1)$, not $(0,1)$.}}
%
We also set
\[
m_{M+1}=1.
\]

%The levels $(m_j)_{j=1}^{M}$ are nondecreasing.  More explicitly, if
%$|v|=0$, then the breakpoints $(a_j)_{j=1}^M$ are $(0,q,1)$ and the levels $(m_j)_{j=1}^{M}$ are $(0,1)$;
%if $0<|v|<q$, then the breakpoints are $(0,|v|,q,1)$ and the levels are
%$(0,0,1)$; if $|v|=q$, then the breakpoints are $(0,q,1)$ and the levels
%are $(0,1)$; if $q<|v|<1$, then the breakpoints are
%(0,q,|v|,1)$ and the levels are $(0,\frac12,1)$; and if $|v|=1$, then
%the breakpoints are $(0,q,1)$ and the levels are $(0,\frac12)$.  We also set
%\[
%m_0=0,\qquad m_{M+1}=1.
%\]

Define a $2\times 2$-matrix $B_j$ by
\begin{equation}\label{eq:defB}
 B_j\coloneqq\beta^2(1-s)q
 \begin{pmatrix}1&1\\1&1\end{pmatrix}
 +s\beta^2Q_j^v
 =\xi_s'(Q_j^v),
 \qquad 0\le j\le M.
\end{equation}
Since $|v|$ is a breakpoint, every increment is of one of the forms
\begin{equation}
 Q_j^v-Q_{j-1}^v=
 \begin{cases}
 (a_j-a_{j-1})
 \begin{pmatrix}1&\iota_v\\\iota_v&1\end{pmatrix},&a_j\le |v|,\\[3mm]
 (a_j-a_{j-1})I_2,&a_{j}> |v|.
 \end{cases}
 \label{eq:GT-overlap-increments}
\end{equation}
Here, note that $a_{j-1} \geq |v|$ if $a_j >|v|$.
Both displayed matrices are positive semidefinite.  Consequently,
\begin{equation}\label{eq:finite-step-vector-increments}
 B_j-B_{j-1}=s\beta^2(Q_j^v-Q_{j-1}^v)\succeq0.
\end{equation}

For a positive semidefinite $2\times2$ matrix $C$, let $G_C$ be a
centered Gaussian vector with covariance $C$, and define
\begin{equation}
 \mathcal T_{m,C}F(x)\coloneqq
 \begin{cases}
 m^{-1}\log\E e^{mF(x+G_C)},&m>0,\\[2mm]
 \E F(x+G_C),&m=0.
 \end{cases}
 \label{eq:GTsemigroup}
\end{equation}
Define functions $U_{s,M}^{\lambda,v}$ by
\begin{equation}
 U_{s,M}^{\lambda,v}=f_\lambda,
 \qquad
 U_{s,j-1}^{\lambda,v}
 =\mathcal T_{m_j,B_j-B_{j-1}}U_{s,j}^{\lambda,v},
 \quad ( 1\le j \le M),
 \label{eq:U-stage-recursion}
\end{equation}
and for $x=(x_1,x_2) \in {\mathbb R}^2$ set
\begin{equation}
 U_s^{\lambda,v}(x)\coloneqq U_{s,0}^{\lambda,v}(x).
\end{equation}
At a breakpoint $a_j$, we use the notation
\begin{equation}\label{eq:U@a}
 U_s^{\lambda,v}(a_j,x)\coloneqq U_{s,j}^{\lambda,v}(x).
\end{equation}

We give another construction of $U_s^{\lambda,v}$ via random variables.
Choose independent centered Gaussian vectors
$G_1,\ldots,G_M$ with $\Var(G_j)=B_j-B_{j-1}$ and let
$\mathcal F_j=\sigma(G_1,\ldots,G_j)$, with $\mathcal F_0$ trivial.
For fixed $x\in\mathbb R^2$, put
\[
 X_M^\lambda=f_\lambda\left(x+\sum_{k=1}^MG_k\right) ,
\]
and for $0\le j \le M$ define $X_j$ recursively by
\begin{equation}
 X_{j-1}^\lambda \coloneqq
 \begin{cases}
 \displaystyle\frac1{m_j}\log\E\left[
 e^{m_jX_j^\lambda}\mid\mathcal F_{j-1}\right],&m_j>0,\\[3mm]
 \displaystyle\E\left[X_j^\lambda\mid\mathcal F_{j-1}\right],&m_j=0.
 \end{cases}
 \label{eq:U-conditional-recursion}
\end{equation}
Then, it holds that
\begin{equation}\label{eq:relationXU}
 X_j^\lambda=U_{s,j}^{\lambda,v}
 \left(x+\sum_{k=1}^jG_k\right),
 \qquad
 X_0^\lambda = U_s^{\lambda,v}(x).
\end{equation}
In particular, $U_s^{\lambda,v}(x)$ is deterministic.  This definition
also covers  the case $s=0$. 
In this case, every covariance increment vanishes. 

Next, put
\begin{equation}\label{eq:defd}
 d_j\coloneqq\frac{s\beta^2}{2}
 \sum_{k,\ell =1}^2(Q_{j,k\ell}^v)^2,
 \qquad 0\le j\le M,
\end{equation}
where $Q_{j,k\ell}^v$ is the $(k,\ell)$ entry of $Q_{j}^v$.
As a function of a breakpoint value $a$, this quantity is given by
\[
 d(a)=
 \begin{cases}
 2s\beta^2a^2,&0\le a\le |v|,\\[1mm]
 s\beta^2(a^2+|v|^2),&|v|\le a\le1.
 \end{cases}
\]
Indeed, it holds that $d_j = d(a_j)$.
Since $d(a)$ is continuous and nondecreasing, 
\begin{equation}\label{eq:finite-step-scalar-increments}
d_j-d_{j-1}\ge0.
\end{equation}
Define
\begin{equation}\label{eq:defLv}
 \mathcal L_s(v)\coloneqq
 \frac12\sum_{j=1}^Mm_j(d_j-d_{j-1}).
\end{equation}
This normalization gives the scalar Gaussian contribution at the
endpoint $t=1$ of the interpolation in Appendix~\ref{app:specialGT}, and is independent of $v$.  Indeed, if $|v|\le q$, then
\[
\sum_{j=1}^Mm_j(d_j-d_{j-1}) = d(1)-d(q)=s\beta^2(1-q^2).
\]
If $q<|v| < 1$,
\begin{align*}
&\sum_{j=1}^Mm_j(d_j-d_{j-1})
 =\frac12(d(|v|)-d(q))+d(1)-d(|v|) \\
& =s\beta^2(|v|^2-q^2)+s\beta^2(1-|v|^2)
 =s\beta^2(1-q^2).
\end{align*}
If $|v|=1$,
\[
 \sum_{j=1}^M m_j(d_j-d_{j-1}) = \frac12(d(1)-d(q))=s\beta^2(1-q^2).
\]
Therefore,
\begin{equation}
 \mathcal L_s(v)=\frac{s\beta^2}{2}(1-q^2),
 \label{eq:GTcorrection}
\end{equation}
In particular, this yields that $\mathcal L_s(v)$ is independent of $v$.

With $Y_0=h+\beta\sqrt{(1-s)q}\,Z$ where $Z$ is a standard normal random variable, set
\begin{equation}
 \mathfrak P_s(\lambda,v)
 \coloneqq2\log2+\E U_s^{\lambda,v}(Y_0,Y_0)
 -\lambda v-\mathcal L_s(v).
 \label{eq:GTfunctional}
\end{equation}
Recall the notation $P_s^*$ from \eqref{eq:RSpathvalue}. For  $v\in\mathcal R_N
\coloneqq\{-1,-1+2/N,\ldots,1\}$, define
\begin{equation}
 Z^{(2)}_{N,s}(v)
 \coloneqq
 \sum_{\substack{\sigma^1,\sigma^2\in\Sigma_N\\R_{12}=v}}
 e^{H_{N,s}(\sigma^1)+H_{N,s}(\sigma^2)}.
 \label{eq:constrainedZ}
\end{equation}

The following finite-volume Guerra--Talagrand upper bound follows from Lemma~\ref{lem:specialGT}, which is proved in Appendix~\ref{app:specialGT}. It is a special case of the general two-dimensional bound obtained by Talagrand and Chen in \cite{Talagrand,ChenGT}:
\begin{equation}
 \frac1N\E\log Z^{(2)}_{N,s}(v)
 \le \inf_{\lambda\in \mathbb R} \mathfrak P_s(\lambda,v).
 \label{eq:GTfunctionalbound}
\end{equation}
\iffalse
The following technical lemma is also postponed to Appendix~\ref{app: technical}.

\begin{lemma}
\label{lem:GTcontinuity}
The maps
\[
 (\beta,h,s,\lambda,v)\longmapsto
 \mathfrak P_s(\lambda,v)-2P_s^*,
 \qquad
 (\beta,h,s,\lambda,v)\longmapsto
 \partial_\lambda\mathfrak P_s(\lambda,v)
\]
are jointly continuous for
$(\beta,h,s,\lambda,v)\in\cK\times[0,1]\times\mathbb R \times[-1,1]$.
\end{lemma}
\fi
\begin{lemma}
\label{lem:U-finite-step-derivatives}
In the recursion \eqref{eq:U-conditional-recursion}, define the tilted
conditional expectation by
\[
\E_j^\lambda[H]
\coloneqq
\frac{
\E\left[
H e^{m_jX_j^\lambda}\mid\mathcal F_{j-1}
\right]
}{
\E\left[
e^{m_jX_j^\lambda}\mid\mathcal F_{j-1}
\right]
},
\qquad m_j>0,
\]
and let $\E_j^\lambda[H]=\E[H\mid\mathcal F_{j-1}]$ when $m_j=0$.
Then
\begin{align}
\partial_\lambda X_{j-1}^\lambda
&=
\E_j^\lambda
\left[
\partial_\lambda X_j^\lambda
\right],
\label{eq:U-first-derivative-recursion}\\
\partial_{\lambda\lambda}X_{j-1}^\lambda
&=
\E_j^\lambda
\left[
\partial_{\lambda\lambda}X_j^\lambda
\right]
+
m_j
\operatorname{Var}_j^\lambda
\left(
\partial_\lambda X_j^\lambda
\right).
\label{eq:U-second-derivative-recursion}
\end{align}
\end{lemma}

\begin{proof}
For $m_j>0$, recalling that $$ X_{j-1}^\lambda= \displaystyle\frac1{m_j}\log\E\left[
 e^{m_jX_j^\lambda}\mid\mathcal F_{j-1}\right],$$
 we have 
\begin{align*}
 \partial_\lambda X_{j-1}^\lambda
 &=\frac{
\E\left[
\partial_\lambda X_j^\lambda e^{m_jX_j^\lambda}\mid\mathcal F_{j-1}
\right]
}{
\E\left[
e^{m_jX_j^\lambda}\mid\mathcal F_{j-1}
\right]}
= \E_j^\lambda[\partial_\lambda X_j^\lambda],\\
 \partial_{\lambda\lambda}X_{j-1}^\lambda
 &=\E_j^\lambda[\partial_{\lambda\lambda}X_j^\lambda]
   +m_j\operatorname{Var}_j^\lambda
       (\partial_\lambda X_j^\lambda).
\end{align*}
For $m_j=0$, direct differentiation of the ordinary conditional
expectation gives the same identities with no variance term.  This completes  the proof. 
\end{proof}

\begin{corollary}\label{cor: technical lemma for U and Psi}
It holds that
    \begin{equation}
\left|\partial_\lambda U_s^{\lambda,v}(x)\right|\le1,
\qquad
0\le
\partial_{\lambda\lambda}U_s^{\lambda,v}(x)
\le\frac52.
\label{eq:U-lambda-derivative-bounds}
\end{equation}
Consequently, uniformly in $(\beta,h,s,v,\lambda)$,
\begin{equation}
0\leq \partial_{\lambda\lambda}\mathfrak P_s(\lambda,v)\leq \frac52.
 \label{eq:GTsecondderivative-simple}
\end{equation}
\end{corollary}
\begin{proof}
Let $\langle\cdot\rangle_{\lambda,x}$ denote expectation on
$\{\pm1\}^2$ with weights proportional to
\[
 \exp(\varepsilon_1x_1+\varepsilon_2x_2
             +\lambda\varepsilon_1\varepsilon_2) \quad ( (\varepsilon _1 , \varepsilon _2) \in \{ \pm 1 \} ^2 ).
\]
Since $(\varepsilon_1\varepsilon_2)^2=1$,
\[
 \partial_\lambda f_\lambda(x)
 =\langle\varepsilon_1\varepsilon_2\rangle_{\lambda,x},
 \qquad
 \partial_{\lambda\lambda}f_\lambda(x)
 =1-\langle\varepsilon_1\varepsilon_2\rangle_{\lambda,x}^2.
\]
Hence, we have $|\partial_\lambda f_\lambda|\le1$ and
$0\le\partial_{\lambda\lambda}f_\lambda\le1$.
By applying \eqref{eq:U-first-derivative-recursion} recursively, we have
\[
 \operatorname{Var}_j^\lambda(\partial_\lambda X_j^\lambda)
 \le \E_j^\lambda[(\partial_\lambda X_j^\lambda)^2]
 \le1.
\]
Similarly, applying \eqref{eq:U-second-derivative-recursion} recursively, we have 
\[
\partial_{\lambda\lambda}X_0^\lambda = \E_1^\lambda \cdots \E_M^\lambda
\left[
\partial_{\lambda\lambda}X_M^\lambda
\right]
+
\sum _{j=1}^M m_j
\E_1^\lambda \cdots \E_{j-1}^\lambda \left[ \operatorname{Var}_j^\lambda
\left(
\partial_\lambda X_j^\lambda
\right) \right] .
\]
This equality, \eqref{eq:relationXU} and the fact that $0\le\partial_{\lambda\lambda}f_\lambda\le1$ yield
\begin{equation}\label{eq:estd2U}
 0\le\partial_{\lambda\lambda}U_s^{\lambda,v}(x)
 \le1+\sum_{j=1}^Mm_j.
\end{equation}
If $|v|\le q$, the sum is at most $1$.  If
$q<|v|<1$, $(m_j)_{j=1}^M = ( 0,\frac12,1)$ and their sum is $3/2$.
If $|v|=1$, $(m_j)_{j=1}^M = ( 0,\frac12)$ and their sum is $1/2$.
Consequently, we have in every case,
\[
 1+\sum_{j=1}^Mm_j\le\frac52.
\]
Together with \eqref{eq:estd2U}, this proves
\eqref{eq:U-lambda-derivative-bounds}.

Moreover, differentiation under the expectation in
\eqref{eq:GTfunctional}
gives
\[
 \partial_{\lambda\lambda}\mathfrak P_s(\lambda,v)
 =\E\partial_{\lambda\lambda}
   U_s^{\lambda,v}(Y_0,Y_0).
\]
Hence \eqref{eq:U-lambda-derivative-bounds} proves
\eqref{eq:GTsecondderivative-simple}.
\end{proof}
\begin{proposition}
\label{prop:GTcoercivity}
Put
\[
 \Gamma_s(v)\coloneqq\inf_{\lambda\in\mathbb R}
 \mathfrak P_s(\lambda,v).
\]
There is a constant $c_{\cK}>0$ such that for $(\beta,h,s,v)\in\cK\times[0,1]\times[-1,1]$,
\begin{equation}
 \Gamma_s(v)
 \le 2P_s^*-c_{\cK}(v-q)^2.
 \label{eq:Gamma-coercivity}
\end{equation}
Moreover,
\begin{equation}
 \Gamma_s(q)=2P_s^*.
 \label{eq:Gamma-at-q}
\end{equation}
\end{proposition}

\begin{proof}

We first evaluate the finite recursion at $\lambda=0$.
Suppose that
$-q\le v\le q$.
In this case, $a_{M-1}=q$, $m_M =1$ and $B_M - B_{M-1} = s\beta ^2 (1-q)I_2$.
Hence, with independent Gaussian random variables $G_1$ and $G_2$ with means $0$ and variances $s\beta ^2(1-q)$, the identities
\[
 f_0(x_1,x_2)=\log\cosh x_1+\log\cosh x_2,
 \qquad
 \left.\partial_\lambda f_\lambda(x_1,x_2)
 \right|_{\lambda=0}=\tanh x_1\tanh x_2 ,
\]
\eqref{eq:GTsemigroup}, \eqref{eq:U-stage-recursion} and \eqref{eq:U@a} give 
\begin{align*}
 U_s^{0,v}(q,x_1,x_2) & = \log \E \left[ \frac{1}{4} \sum _{\varepsilon _1, \varepsilon _2 \in {\pm 1}} \exp (\varepsilon _1 (x_1 + G_1) + \varepsilon _2 (x_2 + G_2))\right] \\
 &=s\beta^2(1-q)+\log\cosh x_1+\log\cosh x_2,
\end{align*}
and
\begin{align*}
 \left.\partial_\lambda U_s^{\lambda,v}(q,x_1,x_2) \right|_{\lambda=0}
 &= \frac{\E \left[ \frac{1}{4} \sum _{\varepsilon _1, \varepsilon _2 \in {\pm 1}} \varepsilon _1 \varepsilon _2 \exp (\varepsilon _1 (x_1 + G_1) + \varepsilon _2 (x_2 + G_2))\right]}{\E \left[ \frac{1}{4} \sum _{\varepsilon _1, \varepsilon _2 \in {\pm 1}} \exp (\varepsilon _1 (x_1 + G_1) + \varepsilon _2 (x_2 + G_2))\right]}\\
 &=\tanh x_1\tanh x_2.
\end{align*}
Hence, by these equalities, the fact that $m_j=0$ for $j<M$, and \eqref{eq:GT-overlap-increments}, we obtain
\begin{align}
U_s^{0,v}(x_1,x_2) &= s\beta^2(1-q) + \E \log\cosh (x_1 + \beta \sqrt{sq}Z_0 ) + \E \log\cosh (x_2 + \beta \sqrt{sq}Z_0 ), \\
\left.\partial_\lambda U_s^{\lambda,v}(x_1,x_2) \right|_{\lambda=0} &= \E \left[ \tanh (x_1 + \beta \sqrt{s|v|}Z_0 + \beta \sqrt{s(q -|v|)} Z_1) \right. \nonumber\\
&\qquad \qquad \left. \times \tanh (x_2 + \iota _v \beta \sqrt{s|v|}Z_0 + \beta \sqrt{s(q-|v|)}Z_2 )\right],
\end{align}
where $Z_0$, $Z_1$ and $Z_2$ are independent standard normal random variables. 
From these equalities it follows that
\begin{align*}
\E U_s^{0,v}(Y_0,Y_0) &= s\beta^2(1-q) + 2 \E \log\cosh (h + \beta \sqrt{q} Z ) \\
\left.\partial_\lambda \E U_s^{\lambda ,v}(Y_0,Y_0) \right|_{\lambda=0} &=\E \left[ \tanh (Y_1(v)) \tanh (Y_2(v) )\right]=:\tilde{g}_s(v),
\end{align*}
where $Y_0 = h +\beta \sqrt{(1-s)q}\,Z$, $Z$ is a standard normal random variable independent from other random variables, and $(Y_1(v),Y_2(v))$ is a  Gaussian pair such that both $Y_1(v)$ and $Y_2(v)$ have means $h$ and variances $\beta^2q$, and have the covariance
\[
 \Cov(Y_1(v),Y_2(v)) =\beta^2\bigl((1-s)q+sv\bigr).
\]
Thus, by applying \eqref{eq:RSpathvalue} and \eqref{eq:GTcorrection}  we have 
\begin{equation}
 \mathfrak P_s(0,v)=2P_s^*,
 \quad
 \partial_\lambda\mathfrak P_s(0,v)= \tilde{g}_s(v)-v , \qquad -q\le v\le q.
 \label{eq:GT-zero-signed-simple}
\end{equation}
%\[
% \tilde{g}_s(v)
% \coloneqq
% \E[\tanh Y_1(v)\tanh Y_2(v)].
%\]
Since Price's theorem (Theorem~\ref{thm:Price}) implies
\begin{align*}
\tilde{g}_s'(v) &= s\beta ^2 \E[\sech ^2 Y_1(v)\sech ^2 Y_2(v)],
\end{align*}
the Cauchy--Schwarz inequality yields
\[
 \tilde{g}_s'(v) \le s\beta ^2 \E[\sech ^4 (h+\beta\sqrt{q}Z)]\le s\alpha,
\]
where $\alpha$ is defined in \eqref{eq:alpha}.
Since $\Cov(Y_1(q),Y_2(q)) = \beta ^2 q = \operatorname{Var} (Y_1(q))$, it holds that $\tilde{g}_s (q) = \E[\tanh ^2 (Y_1(q))] =q$.
Hence, it follows that
\begin{equation}
\tilde{g}_s(v) -v
 \ge(1-s\alpha)(q-v)>0,
 \qquad -q\le v<q.
 \label{eq:GT-signed-positive-simple}
\end{equation}

In  the case $0\le v\le q$, put $t=s\beta^2(q-v)$, let
$X$ be a Gaussian random variable with mean $h$ and variance $\beta^2((1-s)q+sv)$ which is independent of other random variables.  
The variables
\[
 X+\sqrt t\,Z_1,
 \qquad X+\sqrt t\,Z_2,
\]
have the joint law of $(Y_1(v),Y_2(v))$, and are independent under the probability measure conditioned on $X$.
Hence, it follows that
\begin{equation}\label{eq:gv(0<v<q)}
 \tilde{g}_s(v) =\E\left[
 \bigl(\mathsf H_t(\tanh)(X)\bigr)^2
 \right]=g_s(v), \qquad 0\le v \le q
\end{equation}
in view of \eqref{eq:localfieldlaw}, \eqref{eq:Psi-x}, and \eqref{eq:gs-def}.

Now we consider  the case $q<v\le1$. Let $t \coloneqq s\beta^2(v-q)$.
Similarly to  the case $-q\le v\le q$, it holds that for $q< v< 1$
\begin{align*}
 U_s^{0,v}(v,x_1,x_2)
 &=s\beta^2(1-v)+\log\cosh x_1+\log\cosh x_2,\\
 \left.\partial_\lambda U_s^{\lambda,v}(v,x_1,x_2)
 \right|_{\lambda=0}
 & =\tanh x_1\tanh x_2.
\end{align*}
Since $a_{M-1} = v$, $a_{M-2}=q$, $m_{M-1}= \frac{1}{2}$ and 
\[
B_{M-1} - B_{M-2} = s\beta ^2 (v-q)\left( \begin{array}{cc} 1&1\\ 1&1 \end{array}\right) = t \left( \begin{array}{cc} 1&1\\ 1&1 \end{array}\right),
\]
we have% \footnote{\textcolor{red}{The $q<v$ Guerra--Talagrand recursion is missing a factor $2$. Since $m_{M-1}=1/2$, $\mathcal T_{1/2,C}F=2\log\E e^{F/2}$. Thus the logarithmic term below must be $2\log\E\sqrt{\cosh(x_1+\sqrt t\,Z_0)\cosh(x_2+\sqrt t\,Z_0)}$. Without this factor, setting $x_1=x_2=x$ does not yield the subsequently claimed value $2\log\cosh x+s\beta^2(1-q)$; with the factor restored, the diagonal value and the $\lambda$-derivative calculation are consistent.}}
\begin{align*}
 U_s^{0,v}(q,x_1 ,x_2 )
 &=s\beta^2(1-v)
 + 2 \log\E \sqrt{ \cosh(x_1 +\sqrt t\,Z_0) \cosh(x_2+\sqrt t\,Z_0 )} ,
\end{align*}
and 
\begin{align*}
 & \left.\partial_\lambda U_s^{\lambda,v}(q,x_1,x_2) \right|_{\lambda=0} \\
 &= \frac{\E \left[ \left.\partial_\lambda U_s^{\lambda,v}(v,x_1 +\sqrt t\,Z_0 ,x_2 +\sqrt t\,Z_0)
 \right|_{\lambda=0}\exp \left( \frac{1}{2} U_s^{0,v}(v,x_1 +\sqrt t\,Z_0 ,x_2 +\sqrt t\,Z_0 ) \right) \right] }{\E \exp \left( \frac{1}{2} U_s^{0,v}(v,x_1 +\sqrt t\,Z_0 ,x_2 +\sqrt t\,Z_0 ) \right) } \\
 &= \frac{\E \left[ \tanh (x_1 +\sqrt t\,Z_0) \tanh (x_2 +\sqrt t\,Z_0)
 \sqrt{\cosh (x_1 +\sqrt t\,Z_0) \cosh (x_2 +\sqrt t\,Z_0)} \right] }{\E \left[ \sqrt{\cosh (x_1 +\sqrt t\,Z_0) \cosh (x_2 +\sqrt t\,Z_0)} \right] } .
\end{align*}
Note that these equalities also hold  for $v=1$.  
Hence, by letting $x_1=x_2$, we obtain
\begin{align}
U_s^{0,v}(q,x_1 ,x_1) &= 2\log\cosh x_1 +s\beta^2(1-q) , \label{eq:GT-positive-diagonal-value} \\
\left.\partial_\lambda U_s^{\lambda,v}(q,x_1,x_1) \right|_{\lambda=0} &= \frac{\E[\tanh^2(x_1 +\sqrt t\,Z_0 )
                \cosh(x_1 +\sqrt t\,Z_0 )]}
         {\E\cosh(x_1 +\sqrt t\,Z_0 )}
 =\bigl(\mathsf T_t\tanh^2\bigr)(x_1) , \label{eq:GT-positive-diagonal-derivative}
\end{align}
where we applied the fact that $\E\cosh(x+\sqrt t\,Z)=e^{t/2}\cosh x$.
Since $a_1 = q$, $m_1 = 0$ and $B_1 - B_0 = s\beta ^2 q \left( \begin{array}{cc} 1&1\\ 1&1 \end{array}\right)$, from \eqref{eq:GT-positive-diagonal-value} and \eqref{eq:GT-positive-diagonal-derivative} it follows that
\begin{align*}
U_s^{0,v}(x_1 ,x_1) &= 2 \E \log\cosh (x_1 + \beta \sqrt{sq} Z_0) +s\beta^2(1-q) , \\
\left.\partial_\lambda U_s^{\lambda,v}(x_1,x_1) \right|_{\lambda=0} &= \E \bigl(\mathsf T_t\tanh^2\bigr)(x_1 + \beta \sqrt{sq} Z_0) .
\end{align*}
Since these equalities yield
\begin{align*}
 \E U_s^{0,v}(Y_0,Y_0) & = 2\E\log\cosh(h+\beta\sqrt q\,Z_0)+s\beta^2(1-q) , \\
 \E\left. \partial_\lambda U_s^{\lambda,v}(Y_0,Y_0) \right|_{\lambda=0}&= \bigl(\mathsf H_{\beta^2q} \mathsf T_{s\beta^2(v-q)}\tanh^2\bigr)(h)
 =g_s(v),
\end{align*}
where \eqref{eq:Psi-x}, \eqref{eq:localfieldlaw}, and \eqref{eq:gs-def} are applied in the last equality.
Thus, together with \eqref{eq:GT-zero-signed-simple} and \eqref{eq:gv(0<v<q)}, we obtain
\begin{equation}
 \mathfrak P_s(0,v)=2P_s^*,
 \qquad
 \partial_\lambda\mathfrak P_s(0,v)=g_s(v)-v,
 \qquad 0\le v\le1.
 \label{eq:GT-zero-positive-simple}
\end{equation}

We also need a strict gap in  the case $v<-q$.  Suppose first that $s>0$ and let $t=s\beta^2(|v|-q)>0$.
Again, similarly to  the case $-q\le v\le q$, it holds that for $v< -q$
\[
 U_s^{0,v}(|v|,x_1,x_2)
 =s\beta^2(1-|v|)+\log\cosh x_1+\log\cosh x_2.
\]
Hence, since $|v|>q$, an argument similar to the one  above gives  
\[
 U_s^{0,v}(q,x_1,x_2)
 =
 s\beta^2(1-|v|)
 +2\log\E\sqrt{
   \cosh(x_1+\sqrt t\,Z)\cosh(x_2-\sqrt t\,Z)}.
\]
From this equality, the Cauchy--Schwarz inequality and the fact that $\E\cosh(x+\sqrt t\,Z)=e^{t/2}\cosh x$, it follows that
\begin{align*}
 U_s^{0,v}(q,x_1,x_2)
 &\le
 s\beta^2(1-|v|)+t
 +\log\cosh x_1+\log\cosh x_2\\
 &=\log\cosh x_1+\log\cosh x_2 + s\beta^2(1-q).
\end{align*}
%where $\Psi$ is defied in \eqref{eq:Psi-explicit}.
%The formulas remain valid at $v=-1$, where the stage above $r$ has zero covariance.
Since $t>0$, equality holds only when $x_1=-x_2$.
From the fact that $a_1=q$ and $m_1=0$ we have
\[
U_s^{0,v}(x_1,x_2) \le \E \log\cosh (x_1 +\beta\sqrt{sq}\,Z_0) + \E \log\cosh (x_2-\beta\sqrt{sq}\,Z_0) + s\beta^2(1-q).
\]
Since 
\[
\Pbb \left[ Y_0+\beta\sqrt{sq}\,Z_0 = - (Y_0-\beta\sqrt{sq}\,Z_0 ) \right] =\Pbb(Y_0=0) = 0,
\]
where $Y_0 = h +\beta \sqrt{(1-s)q}Z$ and $h>0$, we have the strict inequality:
\[
\E U_s^{0,v}(Y_0,Y_0) < 2 \E \log\cosh (h +\beta\sqrt{q}\,Z) + s\beta^2(1-q) .
\]
This yields
\begin{equation}
 \mathfrak P_s(0,v)<2P_s^*,
 \qquad v<-q,\quad s>0.
 \label{eq:GT-negative-strict-simple}
\end{equation}
When $s=0$, a direct calculation gives
\begin{equation}
 \mathfrak P_0(0,v)=2P_0^*,
 \quad
 \partial_\lambda\mathfrak P_0(0,v)=q-v, \qquad v<-q.
 \label{eq:GT-szero-simple}
\end{equation}

We now prove the desired bound by applying the results above.
Choose
$\varepsilon_{\cK}<q_{\cK}/2$ so that
\eqref{eq:linearATsign} holds when
$|v-q|\le\varepsilon_{\cK}$.  Proposition~\ref{prop:pathRS} gives
$a_{\cK}\in (0,1)$ such that % \footnote{\textcolor{red}{SN: This local quadratic estimate is needed. The sign of $g_s(v)-v$ alone gives only a pointwise gap and does not provide a uniform quadratic gap as $v\to q$.}}
\begin{equation}
 (g_s(v)-v)(v-q)
 \le-a_{\cK}(v-q)^2,
 \qquad |v-q|\le\varepsilon_{\cK}.
 \label{eq:GT-local-slope-simple}
\end{equation}
Since $v\geq 0$ for $|v-q|\le\varepsilon_{\cK}$ with sufficiently small $\varepsilon_{\cK} >0$, \eqref{eq:GTsecondderivative-simple} and \eqref{eq:GT-zero-positive-simple} imply
\[
 \mathfrak P_s(\lambda,v)
 \le
 2P_s^*+\lambda(g_s(v)-v)
 +\frac{5}4 \lambda^2.
\]
Taking $\lambda=2 a_{\cK}(v-q)/5 \in [-1,1]$ and applying \eqref{eq:GT-local-slope-simple} to this inequality yields 
\begin{equation}
 \Gamma_s(v)
 \le
 2P_s^*
 -\frac{a_{\cK}^2}{5}(v-q)^2,
 \qquad |v-q|\le\varepsilon_{\cK}.
 \label{eq:GT-local-coercivity-simple}
\end{equation}

It remains to consider overlaps bounded away from $q$.  We claim that
\begin{equation}
\inf_{\lambda\in [-1,1]}
 \mathfrak P_s(\lambda,v)<2P_s^*
 \qquad\text{for }v\ne q.
 \label{eq:GT-pointwise-gap-simple}
\end{equation}
In  the case $0\le v<q$ and $q<v\le1$, this follows from  \eqref{eq:ATsign} and \eqref{eq:GT-zero-positive-simple}  by moving
$\lambda\in (0,1]$ slightly in the direction opposite to the derivative.
For the case $-q\le v<0$, use \eqref{eq:GT-zero-signed-simple} and \eqref{eq:GT-signed-positive-simple}.
For  the case $v<-q$ and $s>0$, apply \eqref{eq:GT-negative-strict-simple}.
For  the case $v<-q$ and $s=0$, \eqref{eq:GT-szero-simple} gives the claim by taking a sufficiently small negative $\lambda\in [-1,0)$.

% Consider the compact set
% \[
%  \mathcal D_{\cK}
%  =
%  \left\{
%  (\beta,h,s,v):
%  (\beta,h)\in\cK,\ 0\le s\le1,\ -1\le v\le1,\
%  |v-q|\ge\varepsilon_{\cK}
%  \right\}.
% \]
% At every point $\theta\in\mathcal D_{\cK}$,
% \eqref{eq:GT-pointwise-gap-simple} provides
% $\lambda_\theta\in\mathbb R$ such that
% \[
%  \mathfrak P_s(\lambda_\theta,v)<2P_s^*.
% \]
% In view of Lemma~\ref{lem:GTcontinuity}, this inequality persists with a
% positive margin on a neighborhood of $\theta$.
Due to the compactness of $\cK$ and the continuity of $\mathfrak P_s(\lambda,v)$, there exists $\delta _{\cK}>0$ such that
\[
 \Gamma_s(v)\le \inf_{\lambda\in\mathbb [-1,1]}
 \mathfrak P_s(\lambda,v)\le2P_s^*-\delta _{\cK},
 \qquad |v-q|\ge\varepsilon_{\cK}.
\]
Since $|v-q|\le2$,
\begin{equation}
 \Gamma_s(v)
 \le
 2P_s^*-\frac{\delta _{\cK}}4(v-q)^2,
 \qquad |v-q|\ge\varepsilon_{\cK}.
 \label{eq:GT-far-coercivity-simple}
\end{equation}
Combining \eqref{eq:GT-local-coercivity-simple} and
\eqref{eq:GT-far-coercivity-simple} proves
\eqref{eq:Gamma-coercivity}, with
\[
 c_{\cK}
 =
 \min\left\{
 \frac{a_{\cK}^2}{5},
 \frac{\delta _{\cK}}4
 \right\}>0.
\]

Finally, at $v=q$, \eqref{eq:GT-zero-positive-simple} gives
\[
 \mathfrak P_s(0,q)=2P_s^*,
 \qquad
 \partial_\lambda\mathfrak P_s(0,q)=g_s(q)-q=0.
\]
Since \eqref{eq:GTsecondderivative-simple} implies that $\lambda\mapsto\mathfrak P_s(\lambda,q)$ is convex, $\lambda=0$ is
a global minimizer.  Hence
\[
 \Gamma_s(q)=2P_s^*,
\]
which proves \eqref{eq:Gamma-at-q}.
\end{proof}
Together with \eqref{eq:GTfunctionalbound}, the proposition above yields the following. 
\begin{corollary}
        For every $N$ and $v\in\mathcal R_N$,
\begin{equation}
 \frac1N\E\log Z^{(2)}_{N,s}(v)
 \le 2P_s^*-c_{\cK}(v-q)^2.
 \label{eq:GTcoercivity}
\end{equation}
\end{corollary}

\section{Preliminary overlap concentration}\label{sec:concentration}

\subsection{Gaussian calculus}\label{subsec:gaussian}

We record Gaussian integration by parts, which is used below.  Let
$G\coloneqq(G_1,\ldots,G_m)$ be centered Gaussian with covariance
matrix $C$.
If $F$ is continuously differentiable with bounded first derivatives,
then Gaussian integration by parts (see \cite[Vol. 1, Sec. A.4]{Talagrand}) gives
\begin{equation}
 \E[G_iF(G)]=\sum_{j=1}^m C_{ij}\E[\partial_jF(G)].
 \label{eq:GIBP}
\end{equation}

We also use the following standard consequence of Gaussian concentration inequality. 
Let \(G\) be a standard Gaussian vector, and let $\{ F_v\} _{v\in I}$ be a finite set of $L$-Lipschitz functions on the underlying vector space.
Then, it holds that
\begin{equation}
 \mathbb{E}\max_{v\in I}\{F_v (G) -\E F_v (G)\}
 \le L\sqrt{2\log|I|}.
 \label{eq:gaussianmaxgeneral}
\end{equation}
Indeed, the Gaussian concentration inequality (see, for example,
\cite{Ledoux}) gives, for every \(t>0\)
and every \(v\in I\),
\[
 \mathbb{E}\exp\bigl(t(F_v(G) -\E F_v (G))\bigr)
 \le \exp\left(\frac{t^2L^2}{2}\right).
\]
This inequality and Jensen's inequality imply
\[ \begin{aligned}
\mathbb{E}\max_{v\in I}\bigl(F_v(G)-\E F_v(G)\bigr) &\le \frac1t\log \E\exp\left( t\max_{v\in I}\bigl(F_v(G)-\E F_v(G)\bigr) \right)\\ &\le \frac1t\log \sum_{v\in I} \E\exp\left(t\bigl(F_v(G)-\E F_v(G)\bigr)\right)\\
&\le \frac{\log|I|}{t}+\frac{tL^2}{2}, \end{aligned} \] 
where the first inequality follows from Jensen's inequality. 
For $|I|\ge2$ and $L>0$, optimize the upper bound by choosing
$t=\sqrt{2\log|I|}/L$.  The cases $|I|=1$ and $L=0$ are trivial.
Therefore, we obtain \eqref{eq:gaussianmaxgeneral}.
\subsection{Finite-volume Gronwall closure}\label{subsec:gronwall}

For $\rho\ge0$, define the coupled pressure $p^{(2)}_{N,s}(\rho)$ and its normalized excess $F_{N,s}(\rho)$ by
\begin{align}
 p^{(2)}_{N,s}(\rho)
 &\coloneqq
 \frac1{2N}\E\log
 \sum_{\sigma^1,\sigma^2}
 \exp\left(
 H_{N,s}(\sigma^1)+H_{N,s}(\sigma^2)
 +\frac{\rho N}{2}Q_{12}^2\right),
 \label{eq:quadraticpressure}\\
 F_{N,s}(\rho)
 &\coloneqq
 p^{(2)}_{N,s}(\rho)-\phi_N(s)
 =
 \frac1{2N}\E\log
 \braket{e^{\rho NQ_{12}^2/2}}_s ,
 \label{eq:normalizedcoupling}
\end{align}
respectively.
Fix
\begin{equation}
 0<\rho_0<2c_{\cK},
 \label{eq:rho0-choice}
\end{equation}
where $c_{\cK}$ is the coercivity constant in
Proposition~\ref{prop:GTcoercivity}.

\begin{lemma}
\label{lem:sublinearpressure}
There is $C_\cK>0$ such that for any $s\in [0,1]$ and $N\in \mathbb N$, 
\begin{equation}
 p^{(2)}_{N,s}(\rho_0)\le P_s^*+C_{\cK} 
 \sqrt{\frac{\log(N+1)}N}.
 \label{eq:pressureenvelope}
\end{equation}

\end{lemma}

\begin{proof}
Let $L_v\coloneqq\log Z^{(2)}_{N,s}(v)$.  Each $L_v$ is a Lipschitz function
of the Gaussian disorder with squared Lipschitz constant at most
$O_{\cK}(1)N$.  Indeed,
\[
 \left|\partial_{g_{ij}}L_v\right|
 \le\frac{\sqrt2\beta\sqrt s}{\sqrt N},\qquad
 \left|\partial_{z_i}L_v\right|
 \le2\beta\sqrt{(1-s)q},
\]
and summing up the squares over $i,j$ and $i$ gives the claim.  Hence, by \eqref{eq:gaussianmaxgeneral}, 
\begin{equation}
 \mathbb{E}\max_{v\in\mathcal R_N}(L_v-\E L_v)
 \le C_{\cK} \sqrt{N\log(N+1)},
 \label{eq:gaussianmax}
\end{equation}
where we recall $\mathcal R_N \coloneqq \{ -1,-1+2/N,\ldots,1\}$.
Thus, by using \eqref{eq:GTcoercivity}, \eqref{eq:rho0-choice}, and \eqref{eq:gaussianmax}, we have
\begin{align*}
 2Np^{(2)}_{N,s}(\rho_0)
 &=
 \E\log\sum_{v\in\mathcal R_N}
 \exp\left(\E L_v + (L_v - \E L_v ) +\frac{\rho_0N}{2}(v-q)^2\right)\\
 &\le
 \log(N+1)+
 \max_v\left\{\E L_v+\frac{\rho_0N}{2}(v-q)^2\right\}
 +C_{\cK}\sqrt{N\log(N+1)}\\
 &\le2NP_s^*+C_{\cK} 
 \sqrt{{N\log(N+1)}}.
\end{align*}
This proves \eqref{eq:pressureenvelope}.
\end{proof}

Put
\begin{equation}
 D_N(s)\coloneqq P_s^*-\phi_N(s).
 \label{eq:DN}
\end{equation}
\begin{lemma}
\label{lem:preconcentration}
There is $C_\cK>0$ such that for $s\in [0,1]$
\begin{equation}
 D_N(s)\le C_{\cK} 
 \sqrt{\frac{\log(N+1)}N},
 \qquad
 A_s=\nu_s[Q_{12}^2]\le C_{\cK} 
 \sqrt{\frac{\log(N+1)}N}.
 \label{eq:preconcentration}
\end{equation}
\end{lemma}
\begin{proof}
Write 
$$\epsilon_N:= C_{\cK} 
 \sqrt{\frac{\log(N+1)}N}.$$
Similarly to \cite[(14)]{KusuokaNakajima1}, via Gaussian integration by parts (see \cite[Vol.~1,~Sec.~A.4]{Talagrand}), it follows that
\begin{equation}\label{eq:p-derivative}
\phi_N'(s)
=\frac{\beta^2}{4}\left(
(1-q)^2
-
\nu_s[Q_{12}^2]\right).
\end{equation}
%For $0<s<1$, differentiating under the
%expectation and applying \eqref{eq:GIBP} to the variables %$g_{ij}$ and
%$z_i$ gives
%\begin{align*}
% \phi_N'(s)
% ={}&\frac{\beta^2}{4N^2}\sum_{i,j}
% \E\left[1-\braket{\sigma_i\sigma_j}_s^2\right]
% -\frac{\beta^2q}{2N}\sum_i
% \E\left[1-\braket{\sigma_i}_s^2\right]\\
% ={}&\frac{\beta^2}{4}
% \left((1-q)^2-\nu_s[Q_{12}^2]\right).
%\end{align*}
%Here we used
%\[
% \sum_{i,j}\sigma_i^1\sigma_j^1\sigma_i^2\sigma_j^2
% =N^2R_{12}^2,
% \qquad
% \frac1N\sum_i\E\braket{\sigma_i^1\sigma_i^2}_s
% =\nu_s[R_{12}].
%\]
Combining this identity with \eqref{eq:RSpathvalue} and the definition
of $A_s$ in \eqref{eq:ABC} yields
\begin{equation}
 D_N'(s)=\frac{\beta^2}{4}A_s. 
 \label{eq:DNprime}
\end{equation}
%Both sides extend continuously to $s=0,1$ by finite-dimensional
%Gaussian dominated convergence, so the identity also holds for the
%one-sided endpoint derivatives.
On the other hand, convexity of $F_{N,s}(\rho)$ in $\rho \in [0,\infty)$, where $F_{N,s}(\rho)$ is defined in \eqref{eq:normalizedcoupling}, gives
\begin{equation}\label{eq:lemreconcentration1}
   \frac14A_s
 =\left.\partial_\rho F_{N,s}(\rho)\right|_{\rho=0+}
 \le\frac{F_{N,s}(\rho_0)}{\rho_0}.  
\end{equation}
By \eqref{eq:pressureenvelope},
\begin{equation}\label{eq:lemreconcentration2}
 F_{N,s}(\rho_0)
 \le D_N(s)+\epsilon_N.
\end{equation}
Consequently,
\begin{equation}
 D_N'(s)
 \le\frac{\beta^2}{\rho_0}
 \bigl(D_N(s)+\epsilon_N\bigr).
 \label{eq:gronwallDN}
\end{equation}
Gronwall's inequality, the compactness of $\cK$, the continuity of $D_N (s)$ in $s\in [0,1]$, and the fact that $D_N(0)=0$ imply the inequality for $D_N(s)$ in \eqref{eq:preconcentration}.
The estimate for $A_s$ in \eqref{eq:preconcentration} follows from this inequality, \eqref{eq:lemreconcentration1} and \eqref{eq:lemreconcentration2}.
\end{proof}

\subsection{Fixed-deviation overlap concentration}\label{subsec:tails}

\begin{corollary}
\label{cor:tail}
For every $\varepsilon>0$, there are constants
$c_{\cK,\varepsilon},C_{\cK,\varepsilon}>0$ such that
\begin{equation}
 \sup_{\substack{(\beta,h)\in\cK\\0\le s\le1}}
 \nu_s\bigl[|Q_{12}|\ge\varepsilon\bigr]
 \le C_{\cK,\varepsilon}e^{-c_{\cK,\varepsilon}N}.
 \label{eq:tail}
\end{equation}
\end{corollary}

\begin{proof}
Let
\[
 Y_{N,s}
 \coloneqq\log\braket{e^{\rho_0NQ_{12}^2/2}}_s.
\]
By \eqref{eq:normalizedcoupling}, \eqref{eq:preconcentration} and \eqref{eq:lemreconcentration2},
\begin{equation}\label{eq:EY}
 \E Y_{N,s}=2NF_{N,s}(\rho_0)
 \le C_{\cK} N\epsilon_N=o(N)
\end{equation}
uniformly in $s\in [0,1]$.
It is known that the map
\begin{align*}
((g_{ij})_{i,j}, (z_i)_{i}) \mapsto Y_{N,s} = & \log
 \sum_{\sigma^1,\sigma^2}
 \exp\left(
 H_{N,s}(\sigma^1)+H_{N,s}(\sigma^2)
 +\frac{\rho _0 N}{2}Q_{12}^2\right) \\
&- \log
 \sum_{\sigma^1,\sigma^2}
 \exp\left(
 H_{N,s}(\sigma^1)+H_{N,s}(\sigma^2)
 \right) ,
\end{align*}
where we note that $H_{N,s}$ depends on $((g_{ij})_{i,j}, (z_i)_{i})$, is Lipschitz continuous with Lipschitz constant $C_{\cK}\sqrt{N}$
(see the proof of Lemma~\ref{lem:sublinearpressure} or \cite[Theorem~1.3.5]{Talagrand} and the discussion after the theorem).
The Gaussian concentration inequality
(see \cite[Theorem~1.3.4]{Talagrand})
%(see, for example,\cite{Ledoux})
therefore gives, for all
large enough $N$,
\[
 \mathbb P\left(
 |Y_{N,s} - \mathbb E Y_{N,s}|> t \right)
\le \exp \left( - N^{-1 } t^2/C_{\cK}  \right) \qquad (t\in [0,\infty )).
\]
This inequality and \eqref{eq:EY} yield
\begin{equation}\label{eq:tailY}
\mathbb P\left(
 Y_{N,s}>\frac{\rho_0\varepsilon^2N}{4}\right) \le C_{\cK,\varepsilon}e^{-c_{\cK,\varepsilon}N}.
\end{equation}
For each realization of the disorder, on $\{|Q_{12}|\ge\varepsilon\}$ we have
\(
\1_{\{|Q_{12}|\ge\varepsilon\}}
\le
\exp\left(
\frac{\rho_0N}{2}(Q_{12}^2-\varepsilon^2)
\right).
\) 
Therefore,
\[
\left\langle \1_{\{|Q_{12}|\ge\varepsilon\}}\right\rangle_s
\le
\exp\left(-\frac{\rho_0\varepsilon^2N}{2}\right)
\left\langle
\exp\left(\frac{\rho_0NQ_{12}^2}{2}\right)
\right\rangle_s
=
\exp\left(
-\frac{\rho_0\varepsilon^2N}{2}+Y_{N,s}
\right).
\]
Hence, 
\eqref{eq:tailY} implies
\begin{align*}
\nu_s\bigl[|Q_{12}|\ge\varepsilon\bigr]
&\le
\E\left[
\exp\left(
-\frac{\rho_0\varepsilon^2N}{2}+Y_{N,s}
\right)
\1_{\left\{Y_{N,s}\le \frac{\rho_0\varepsilon^2N}{4}\right\}}
\right]
+
\mathbb P\left(
Y_{N,s}>\frac{\rho_0\varepsilon^2N}{4}
\right)
\\
&\le
\exp\left(-\frac{\rho_0\varepsilon^2N}{4}\right)
+
\mathbb P\left(
Y_{N,s}>\frac{\rho_0\varepsilon^2N}{4}
\right).
\end{align*}
Therefore, we obtain \eqref{eq:tail}.
\end{proof}
\section{Cavity approximations}
\label{app:cavity-identities}

Recall from \eqref{eq:ABC} that
\[
 A_s=\nu_s[Q_{12}^2],\qquad
 B_s=\nu_s[Q_{12}Q_{13}],\qquad
 C_s=\nu_s[Q_{12}Q_{34}].
\]
Introduce the three cavity modes
\begin{equation}
 U_s\coloneqq A_s-4B_s+3C_s,\qquad
 V_s\coloneqq2B_s-3C_s,\qquad
 D_s\coloneqq A_s-2B_s+C_s,
 \label{eq:cavity-modes}
\end{equation}
and put
\begin{equation}
 \kappa \coloneqq1-4q+3r,
 \qquad
 \zeta\coloneqq2q+q^2-3r.
 \label{eq:cavity-parameters}
\end{equation}

\begin{proposition}
\label{prop:cavity}
There are remainders
$\mathcal R^U_{N,s},\mathcal R^V_{N,s},\mathcal R^D_{N,s}$ and $C_\cK>0$ such that
\begin{align}
 (1-s\beta^2\kappa )U_s
 &=\frac{\kappa }N+\mathcal R^U_{N,s},
 \label{eq:cavity-U}\\
 (1-s\beta^2\kappa )V_s-s\beta^2\zeta U_s
 &=\frac{\zeta}N+\mathcal R^V_{N,s},
 \label{eq:cavity-V}\\
 (1-s\alpha)D_s
 &=\frac{\alpha}{\beta^2N}+\mathcal R^D_{N,s},
 \label{eq:cavity-D}
\end{align}
and uniformly on $\cK\times[0,1]$,
\begin{equation}
 \max_{X\in\{U,V,D\}}|\mathcal R^X_{N,s}|
 \le C_{\cK}(N^{-3/2}+\nu_s[|Q_{12}|^3]).
 \label{eq:cavity-remainders}
\end{equation}
\end{proposition}
\begin{proof}
 All implicit constants in this proof depend only on $\cK$.  Write
\[
\varepsilon_a\coloneqq\sigma_N^a,
 \quad
 Q^-_{ab}\coloneqq\frac1N\sum_{i<N}\sigma_i^a\sigma_i^b-q,
\]
so that
\begin{equation}
 Q_{ab}=Q^-_{ab}+\frac1N\varepsilon_a\varepsilon_b.
 \label{eq:Qab}
\end{equation}
By site exchangeability after averaging over the disorder,
\begin{align*}
 A_s&=\nu_s\!\left[Q_{12}(\varepsilon_1\varepsilon_2-q)\right],&
 B_s&=\nu_s\!\left[Q_{12}(\varepsilon_1\varepsilon_3-q)\right],&
 C_s&=\nu_s\!\left[Q_{12}(\varepsilon_3\varepsilon_4-q)\right].
\end{align*}

Use the last-spin cavity interpolation.  Write
$\sigma=(\rho,\varepsilon)$, with $\rho\in\{-1,1\}^{N-1}$, and replace
the interaction of the last spin with the first $N-1$ spins by
\[
 \frac{\beta\sqrt{su}}{\sqrt N}
 \sum_{i<N}\frac{g_{iN}+g_{Ni}}{\sqrt2}\rho_i\varepsilon
 +\beta\sqrt{s(1-u)q}\,\widehat z\,\varepsilon,
 \qquad 0\le u\le1,
\]
where $\widehat z$ is an independent standard Gaussian.  Denote the
corresponding averaged Gibbs expectation by $\nu_{s,u}$.  Then
$\nu_{s,1}=\nu_s$.  At $u=0$, the last spin is independent of the first
$N-1$ spins and is subject to the field
\[
 h+\beta\sqrt{(1-s)q}\,z_N+\beta\sqrt{sq}\,\widehat z
 \stackrel{\mathrm d}=h+\beta\sqrt q\,Z.
\]
Thus, for distinct replica indices,
\begin{equation}
 \mathbb E_0[\varepsilon_a\varepsilon_b]=q,
 \qquad
 \mathbb E_0[\varepsilon_a\varepsilon_b\varepsilon_c\varepsilon_d]=r.
 \label{eq:last-spin-moments-proof}
\end{equation}

Gaussian integration by parts gives, for every bounded function $F$ of
$n$ replicas,
\begin{align}
 \frac{\mathrm d}{\mathrm du}\nu_{s,u}[F]
 =s\beta^2\Bigg\{&
 \sum_{1\le a<b\le n}
 \nu_{s,u}[F\varepsilon_a\varepsilon_bQ^-_{ab}]
 -n\sum_{a=1}^n
 \nu_{s,u}[F\varepsilon_a\varepsilon_{n+1}Q^-_{a,n+1}]
 \nonumber\\
 &+\frac{n(n+1)}2
 \nu_{s,u}[F\varepsilon_{n+1}\varepsilon_{n+2}
 Q^-_{n+1,n+2}]
 \Bigg\}.
 \label{eq:cavity-derivative-proof}
\end{align}
Indeed, the derivative of the covariance of the interpolating
last-spin field is
$s\beta^2\varepsilon\varepsilon'Q^-(\rho,\rho')$, while its diagonal
value is independent of the configuration.

Set $M_3(u)\coloneqq\nu_{s,u}[|Q^-_{12}|^3]$.  Formula
\eqref{eq:cavity-derivative-proof} and $|Q^-_{ab}|\le2$ give
\[
 |M_3'(u)|\le C_{\cK}M_3(u).
\]
Since
\[
 |Q^-_{12}|^3\le4|Q_{12}|^3+\frac4{N^3},
\]
Gronwall's inequality, applied backward from $u=1$, yields
\begin{equation}
 \sup_{0\le u\le1}\nu_{s,u}[|Q^-_{12}|^3]
 \le C_{\cK}(N^{-3/2}+\nu_s[|Q_{12}|^3]).
 \label{eq:M3-gronwall-proof}
\end{equation}

For
\[
 P^-=Q^-_{ab}Q^-_{cd},
 \qquad (ab,cd)\in\{(12,12),(12,13),(12,34)\},
\]
the derivative of $\nu_{s,u}[P^-]$ is a finite linear combination of
expectations of three cavity overlaps.  Hence H\"older's inequality and
\eqref{eq:M3-gronwall-proof} show that
\[
 |\nu_{s,0}[P^-]-\nu_{s,1}[P^-]|
 \le C_{\cK}(N^{-3/2}+\nu_s[|Q_{12}|^3]).
\]
Moreover, \eqref{eq:Qab} gives
\[
 |P^--Q_{ab}Q_{cd}|
 \le\frac{|Q_{ab}|+|Q_{cd}|}{N}+\frac1{N^2}.
\]
Using
\[
 \frac{|x|}{N}\le\frac{|x|^3}{3}+\frac{2}{3N^{3/2}},
\]
we obtain
\begin{equation}
 \left|\nu_{s,0}[P^-]-\nu_s[Q_{ab}Q_{cd}]\right|
 \le C_{\cK}(N^{-3/2}+\nu_s[|Q_{12}|^3]).
 \label{eq:endpoint-replacement-proof}
\end{equation}

Define
\begin{align*}
 A_s^-&\coloneqq\nu_{s,0}[(Q^-_{12})^2],&
 B_s^-&\coloneqq\nu_{s,0}[Q^-_{12}Q^-_{13}],&
 C_s^-&\coloneqq\nu_{s,0}[Q^-_{12}Q^-_{34}],
\end{align*}
and form $U_s^-,V_s^-,D_s^-$ from $A_s^-,B_s^-,C_s^-$ as in
\eqref{eq:cavity-modes}.  It follows from
\eqref{eq:endpoint-replacement-proof} that
\begin{equation}
 |U_s^--U_s|+|V_s^--V_s|+|D_s^--D_s|
 \le C_{\cK}(N^{-3/2}+\nu_s[|Q_{12}|^3]).
 \label{eq:cavity-mode-replacement}
\end{equation}

Next write
\begin{align}
 A_s&=\nu_s[G_A]+\frac1N\nu_s[E_A],\nonumber\\
 B_s&=\nu_s[G_B]+\frac1N\nu_s[E_B],\nonumber\\
 C_s&=\nu_s[G_C]+\frac1N\nu_s[E_C],
 \label{eq:cavity-decomposition-proof}
\end{align}
where
\begin{align*}
 G_A&\coloneqq Q^-_{12}(\varepsilon_1\varepsilon_2-q),&
 E_A&\coloneqq\varepsilon_1\varepsilon_2
 (\varepsilon_1\varepsilon_2-q),\\
 G_B&\coloneqq Q^-_{12}(\varepsilon_1\varepsilon_3-q),&
 E_B&\coloneqq\varepsilon_1\varepsilon_2
 (\varepsilon_1\varepsilon_3-q),\\
 G_C&\coloneqq Q^-_{12}(\varepsilon_3\varepsilon_4-q),&
 E_C&\coloneqq\varepsilon_1\varepsilon_2
 (\varepsilon_3\varepsilon_4-q).
\end{align*}
Put
\[
 g_j(u)\coloneqq\nu_{s,u}[G_j],
 \qquad
 e_j(u)\coloneqq\nu_{s,u}[E_j],
 \qquad j\in\{A,B,C\},
\]
and take the same linear combinations as in \eqref{eq:cavity-modes}:
\begin{align*}
 g_U&\coloneqq g_A-4g_B+3g_C,&
 g_V&\coloneqq2g_B-3g_C,&
 g_D&\coloneqq g_A-2g_B+g_C,
\end{align*}
with analogous definitions for $e_U,e_V,e_D$.

Independence at $u=0$ gives $g_j(0)=0$.  Applying
\eqref{eq:cavity-derivative-proof} twice, H\"older's inequality and
\eqref{eq:M3-gronwall-proof} give
\begin{equation}
 g_X(1)=g_X'(0)+C_{\cK}(N^{-3/2}+\nu_s[|Q_{12}|^3]),
 \qquad X\in\{U,V,D\}.
 \label{eq:cavity-Taylor-modes}
\end{equation}
For a replica edge $\{a,b\}$, \eqref{eq:last-spin-moments-proof}
implies
\begin{equation}
 \mathbb E_0[(\varepsilon_p\varepsilon_q-q)
 \varepsilon_a\varepsilon_b]
 =
 \begin{cases}
 1-q^2,&\{a,b\}=\{p,q\},\\
 q-q^2,&\{a,b\}\text{ shares one endpoint with }\{p,q\},\\
 r-q^2,&\{a,b\}\cap\{p,q\}=\varnothing.
 \end{cases}
 \label{eq:edge-rule-proof}
\end{equation}
Substitution into \eqref{eq:cavity-derivative-proof}, followed directly
by the three linear combinations above, gives
\begin{align}
 g_U'(0)&=s\beta^2\kappa  U_s^-,\nonumber\\
 g_V'(0)&=s\beta^2(\zeta U_s^-+\kappa V_s^-),
 \label{eq:scalar-first-derivatives}\\
 g_D'(0)&=s\alpha D_s^- .\nonumber
\end{align}

The last-spin moments also give
\begin{align}
 e_U(0)&=(1-q^2)-4(q-q^2)+3(r-q^2)=\kappa ,\nonumber\\
 e_V(0)&=2(q-q^2)-3(r-q^2)=\zeta,
 \label{eq:scalar-diagonal-terms}\\
 e_D(0)&=1-2q+r=\frac{\alpha}{\beta^2}.
 \nonumber
\end{align}
Since $E_A,E_B,E_C$ are uniformly bounded,
\eqref{eq:cavity-derivative-proof} yields
\[
 |e_X'(u)|\le C_{\cK}\nu_{s,u}[|Q^-_{12}|],
 \qquad X\in\{U,V,D\}.
\]
The inequality
\(
 \frac{|Q^-_{12}|}{N}
 \le\frac{|Q^-_{12}|^3}{3}+\frac{2}{3N^{3/2}}
\) 
and \eqref{eq:M3-gronwall-proof} therefore imply
\begin{equation}
 \frac1N|e_X(1)-e_X(0)|
 \le C_{\cK}(N^{-3/2}+\nu_s[|Q_{12}|^3]),
 \qquad X\in\{U,V,D\}.
 \label{eq:scalar-diagonal-remainders}
\end{equation}

From
\eqref{eq:cavity-mode-replacement}, \eqref{eq:cavity-decomposition-proof}, 
\eqref{eq:cavity-Taylor-modes}, \eqref{eq:scalar-first-derivatives},
\eqref{eq:scalar-diagonal-terms}, and
\eqref{eq:scalar-diagonal-remainders}, it follows that
\begin{align*}
 U_s&=s\beta^2\kappa U_s+\frac{\kappa }N
 +C_{\cK}(N^{-3/2}+\nu_s[|Q_{12}|^3]),\\
 V_s&=s\beta^2(\zeta U_s+\kappa V_s)+\frac{\zeta}N
 +C_{\cK}(N^{-3/2}+\nu_s[|Q_{12}|^3]),\\
 D_s&=s\alpha D_s+\frac{\alpha}{\beta^2N}
 +C_{\cK}(N^{-3/2}+\nu_s[|Q_{12}|^3]).
\end{align*}
These yield \eqref{eq:cavity-U}--\eqref{eq:cavity-remainders}.
\end{proof}
Since $r\le q$,
\[
 \alpha-\beta^2\kappa=2\beta^2(q-r)\ge0.
\]
Thus $\beta^2\kappa\le\alpha<1$.  If $\kappa<0$, then
$1-s\beta^2\kappa\ge1$, while if $\kappa\ge0$, then
$1-s\beta^2\kappa\ge1-\alpha$.  Hence \eqref{eq:delta} gives, uniformly
on $\cK\times[0,1]$,
\begin{equation}
 1-s\alpha\ge\delta_{\cK},\qquad
 1-s\beta^2\kappa\ge\delta_{\cK}.
 \label{eq:denominator-bounds}
\end{equation}
Solving \eqref{eq:cavity-U}--\eqref{eq:cavity-D} in the order
$U_s,V_s,D_s$ yields
\[
 |U_s|+|V_s|+|D_s|
 \le C_{\cK}\left(\frac1N+N^{-3/2}+\nu_s[|Q_{12}|^3]\right).
\]
Since $A_s=-V_s-2U_s+3D_s$, it follows that
\begin{equation}
\nu_s[|Q_{12}|^2] = A_s\le\frac{C_{\cK}}N+C_{\cK}\nu_s[|Q_{12}|^3].
 \label{eq:preabsorb}
\end{equation}

\section{Quantitative bounds for overlaps, free energy, and susceptibility}\label{sec:conclusion}

\subsection{Absorption of the cubic remainder and the proof of \eqref{eq:main1}}\label{subsec:absorb}

Fix $\varepsilon>0$.  Since $|Q_{12}|\le2$,
\begin{align}
 \nu_s[|Q_{12}|^3]
 &\le
 \varepsilon\,\nu_s[Q_{12}^2]
 +8\,\nu_s[|Q_{12}|>\varepsilon]\nonumber\\
 &\le
 \varepsilon A_s
 +8C_{\cK,\varepsilon}e^{-c_{\cK,\varepsilon}N},
 \label{eq:thirdsplit}
\end{align}
where Corollary~\ref{cor:tail} was used.  Choose
$\varepsilon=\varepsilon_{\cK}=1/(2C_\cK)$
so that the coefficient of $A_s$ obtained by substituting
\eqref{eq:thirdsplit} into \eqref{eq:preabsorb} is at most $1/2$.
We obtain
\[
 A_s=\nu_s[Q_{12}^2] \le\frac{C_{\cK}'}N,
\]
with some $C_\cK'>0$ uniformly on $\cK\times[0,1]$.  This proves \eqref{eq:main1}. 

\subsection{Quantitative estimate of the free energy: Proof of \eqref{eq:main2}}\label{subsec:freeenergy}

Recall that $D_N(s):= P_s^* -\phi_N(s)$. The identity \eqref{eq:DNprime}, $D_N'(s) = \frac{\beta^2}{4}\nu_s[Q_{12}^2]$, with $P_1^*  =  \phi^{\mathrm{RS}}(\beta,h)$ and  $D_N(0)=0$ gives
\begin{equation}
 \phi^{\mathrm{RS}}(\beta,h)-\phi_N(1)
 =D_N(1)=
 \frac{\beta^2}{4}\int_0^1\nu_s[Q_{12}^2]\,\dd s.
 \label{eq:freeenergyidentity}
\end{equation}
The integrand is nonnegative, and the  bound above implies that the right-hand side is at most $C_{\cK}/N$.  This proves \eqref{eq:main2}.

\subsection{Replicon susceptibility: Proof of \eqref{eq:main3}}\label{subsec:susceptibility}

The $O(N^{-1})$ second-moment estimate and
\eqref{eq:thirdsplit} imply, for every fixed $\varepsilon>0$,
\[
 \sup_{\cK\times[0,1]}
 N\nu_s[|Q_{12}|^3]
 \le C_{\cK}\varepsilon
 +8N C_{\cK,\varepsilon}e^{-c_{\cK,\varepsilon}N}.
\]
First let $N\to\infty$ and then $\varepsilon\downarrow0$.  Thus
\begin{equation}\label{eq:Nnu3}
 \nu_s[|Q_{12}|^3]=o_{\cK}(N^{-1})
\end{equation}
uniformly in $s$.  Equation \eqref{eq:cavity-D}, together with
\eqref{eq:cavity-remainders} and \eqref{eq:denominator-bounds}, now gives
\[
 ND_s
 =
 \frac{\alpha}{\beta^2(1-s\alpha)}+o_{\cK}(1).
\]
This proves \eqref{eq:main3} and completes the proof of
Theorem~\ref{thm:main}.

% This file is intended to be inserted immediately after Appendix A.
% It contains exactly one section and uses the notation already introduced
% in the main text and in the cavity appendix.

\section{Proof of the central limit theorem for the overlap}
\label{sec:overlap-clt}
We follow a similar argument to \cite[Chapter 1]{Talagrand}. In this section we work at $s=1$ and suppress the subscript $s$.  The argument above \eqref{eq:denominator-bounds} gives
$\beta^2\kappa\le\alpha<1$, and hence
\begin{equation}
 1-\alpha>0,\qquad
 1-\beta^2\kappa>0.
 \label{eq:clt-positive-denominators}
\end{equation}
Define
\begin{equation}
 \sigma^2
 \coloneqq
 \frac{3\alpha}{\beta^2(1-\alpha)}
 -\frac{2\kappa}{1-\beta^2\kappa}
 -\frac{\zeta}{(1-\beta^2\kappa)^2}.
 \label{eq:clt-variance}
\end{equation}

We verify that $\sigma^2>0$. Set
$$
w\coloneqq\frac{\beta^2(q-r)}{\alpha}.
$$
Note that $0<\alpha <1$, $q^2\leq  r\leq q$, and $w \geq 0$ due to our assumptions.  
Moreover, since 
$$
\kappa
=1-4q+3r
= \frac{\alpha}{\beta^2} -2(q-r),\quad \zeta
=2q+q^2-3r
=2(q-r)+(q^2-r)
\le 2(q-r),
$$
we have 
$$
\frac{\beta^2 \kappa}{\alpha}=1-2w,\quad \frac{\beta^2 \zeta}{\alpha}\le2w.
$$
Thus, $\beta^2 \kappa \leq \alpha<1$ and 
we obtain
$$
\frac{\beta^2 \sigma^2}{\alpha}
=
\frac{3}{1-\alpha}
-\frac{2(1-2w)}{1-\beta^2\kappa}
-\frac{\beta^2\zeta}{\alpha(1-\beta^2\kappa)^2}
\ge
\frac{3}{1-\alpha}
-\frac{2(1-2w)}{1-\beta^2\kappa}
-\frac{2w}{(1-\beta^2\kappa)^2}.
$$
Using $1-\beta^2\kappa=1-\alpha+2\alpha w$,
$$
\frac{3}{1-\alpha}
-\frac{2(1-2w)}{1-\beta^2\kappa}
-\frac{2w}{(1-\beta^2\kappa)^2}
=
\frac{
(1-\alpha)^2
+2(1-\alpha)(1+2\alpha)w
+4\alpha(\alpha+2)w^2
}{
(1-\alpha)(1-\beta^2\kappa)^2
}
\ge
\frac{1-\alpha}{(1-\beta^2\kappa)^2}.
$$
Therefore
$$
\sigma^2
\ge
\frac{\alpha(1-\alpha)}{\beta^2 (1-\beta^2\kappa)^2}
>0.
$$

Now we give the proof of Theorem~\ref{thm:overlap-clt-intro}.

\begin{proof}[Proof of Theorem~\ref{thm:overlap-clt-intro}]
Let $\nu_u$, $0\le u\le1$, be the last-spin interpolation from
Section~\ref{app:cavity-identities}.  We write $\nu= \nu_1$, and
\[
 \varepsilon_a:=\sigma_N^a,\qquad
 Q^-_{ab}:=\frac1N\sum_{i<N}\sigma_i^a\sigma_i^b-q,\qquad
 Q_{ab}:=Q^-_{ab}+\frac1N\varepsilon_a\varepsilon_b.
\]
The overlap estimate \eqref{eq:main1} gives
\begin{equation}
 \nu[Q_{12}^2]\le\frac{C}{N}.
 \label{eq:clt-endpoint-second-moment}
\end{equation}
For $M_2(u)=\nu_u[(Q^-_{12})^2]$,
\eqref{eq:cavity-derivative-proof} and $|Q^-_{ab}|\le2$ imply
\begin{equation}
 |M_2'(u)|\le C M_2(u).
 \label{eq:clt-M2-differential}
\end{equation}
Since $M_2(1)\le2\nu[Q_{12}^2]+2N^{-2}$, the backward Gronwall inequality and \eqref{eq:clt-endpoint-second-moment} yield
\begin{equation}
 \sup_{0\le u\le1}\nu_u[(Q^-_{12})^2]\le\frac{C}{N}.
 \label{eq:clt-uniform-cavity-second}
\end{equation}
The Cauchy--Schwarz inequality therefore gives, for fixed edges
$ab,cd$,
\begin{align}
 \sup_{0\le u\le1}\nu_u[|Q^-_{ab}|]
 \le\frac{C}{\sqrt N},
\quad  \sup_{0\le u\le1}\nu_u[|Q^-_{ab}Q^-_{cd}|]
 \le\frac{C}{N}.
 \label{eq:clt-cavity-product}
\end{align}

Fix $f\in C^2_b(\mathbb R;\mathbb C)$, i.e. $\max_{n\in \{0,1,2\}}\|f^{(n)}\|_\infty<\infty$, and set
$X_{ab}:=\sqrt N\,Q_{ab}$.  Define
\begin{align}
 T_A(f)&\coloneqq\nu[X_{12}f(X_{12})],
 \label{eq:clt-TA}\\
 T_B(f)&\coloneqq\nu[X_{13}f(X_{12})],
 \label{eq:clt-TB}\\
 T_C(f)&\coloneqq\nu[X_{34}f(X_{12})].
 \label{eq:clt-TC}
\end{align}
Site exchangeability gives
\begin{equation}
 T_E(f)=\sqrt N\,\nu[
 \overline\varepsilon_{e(E)}f(X_{12})],
 \label{eq:clt-site-exchangeability}
\end{equation}
where $\overline\varepsilon_{ab}\coloneqq\varepsilon_a\varepsilon_b-q$, $e(A)\coloneqq 12$, $e(B)\coloneqq 13$ and $e(C)\coloneqq 34$. 
Put
\[
 g_E(u)\coloneqq
 \sqrt N\,\nu_u[\overline\varepsilon_{e(E)}f(X_{12})],
 \qquad E\in\{A,B,C\}.
\]
Then $g_E(1)=T_E(f)$.  We write $O_f(\cdot)$ for a bound whose constant
depends only on $\beta,h$ and
$\|f\|_\infty+\|f'\|_\infty+\|f''\|_\infty$.

Remove the last-spin contribution from the argument of $f$ by setting
\begin{equation}
 Y\coloneqq\sqrt N\,Q^-_{12}+\frac q{\sqrt N}.
 \label{eq:clt-Yminus}
\end{equation}
Then
\begin{equation}
 X_{12}=Y+\frac{\overline\varepsilon_{12}}{\sqrt N}.
 \label{eq:clt-X-Y}
\end{equation}
At $u=0$, \eqref{eq:last-spin-moments-proof} gives
\begin{align}
 \mathbb E_0[\overline\varepsilon_{12}^2]
 &=1-q^2,\,
 \mathbb E_0[\overline\varepsilon_{13}\overline\varepsilon_{12}]
 =q-q^2,\,
 \mathbb E_0[\overline\varepsilon_{34}\overline\varepsilon_{12}]
 =r-q^2.
\end{align}
Taylor expansion in \eqref{eq:clt-X-Y} and independence at $u=0$ imply
\begin{equation}
 g_E(0)=\theta_E\nu_0[f'(Y)]+O_f(N^{-1/2}),
 \label{eq:clt-g-zero-pre}
\end{equation}
where
\begin{equation}
 \theta_A=1-q^2,\qquad
 \theta_B=q-q^2,\qquad
 \theta_C=r-q^2.
 \label{eq:clt-theta}
\end{equation}
By \eqref{eq:cavity-derivative-proof} and
\eqref{eq:clt-cavity-product},
\[
 \left|\frac{\mathrm d}{\mathrm du}\nu_u[f'(Y)]\right|
 \le\frac{C\|f'\|_\infty}{\sqrt N}.
\]
Together with \eqref{eq:clt-X-Y}, this gives 
\begin{equation}
 \nu_0[f'(Y)]=\nu[f'(X_{12})]+O_f(N^{-1/2}).
 \label{eq:clt-derivative-endpoint-replacement}
\end{equation}
Hence, in view of \eqref{eq:clt-g-zero-pre} it holds that
\begin{equation}
 g_E(0)=\theta_E m_f+O_f(N^{-1/2}),
 \qquad m_f\coloneqq\nu[f'(X_{12})].
 \label{eq:clt-g-zero}
\end{equation}

Define
\begin{align}
 \widetilde T_A(f)
 &\coloneqq\sqrt N\,\nu_0[Q^-_{12}f(Y)],
 \label{eq:clt-Ttilde-A}\\
 \widetilde T_B(f)
 &\coloneqq\sqrt N\,\nu_0[Q^-_{13}f(Y)],
 \label{eq:clt-Ttilde-B}\\
 \widetilde T_C(f)
 &\coloneqq\sqrt N\,\nu_0[Q^-_{34}f(Y)].
 \label{eq:clt-Ttilde-C}
\end{align}
Since we can replace $f(X_{12})$ by $f(Y)$ with error
$O_f(N^{-1/2})$ due to 
\eqref{eq:clt-cavity-product} and \eqref{eq:clt-X-Y},   a similar calculation to 
\eqref{eq:scalar-first-derivatives} using \eqref{eq:cavity-derivative-proof} specialized to $s=1$, gives
\begin{align}
 g_A'(0)&=\beta^2\bigl(
 \theta_A\widetilde T_A-4\theta_B\widetilde T_B
 +3\theta_C\widetilde T_C\bigr)+O_f(N^{-1/2}),
 \label{eq:clt-row-A}\\
 g_B'(0)&=\beta^2\bigl(
 \theta_B\widetilde T_A
 +(\theta_A-2\theta_B-3\theta_C)\widetilde T_B
 +(6\theta_C-3\theta_B)\widetilde T_C\bigr)
 +O_f(N^{-1/2}),
 \label{eq:clt-row-B}\\
 g_C'(0)&=\beta^2\bigl(
 \theta_C\widetilde T_A
 +(4\theta_B-8\theta_C)\widetilde T_B
 +(\theta_A-8\theta_B+10\theta_C)\widetilde T_C\bigr)
 +O_f(N^{-1/2}).
 \label{eq:clt-row-C}
\end{align}
Denote $T=(T_A,T_B,T_C)^{\mathsf T}$, $\widetilde T=(\widetilde T_A,\widetilde T_B,\widetilde T_C)^{\mathsf T}$ and $\theta=(\theta_A,\theta_B,\theta_C)^{\mathsf T}$, and let
\begin{equation}
 \mathcal M_1
 :=\beta^2
 \begin{pmatrix}
 \theta_A&-4\theta_B&3\theta_C\\
 \theta_B&\theta_A-2\theta_B-3\theta_C&6\theta_C-3\theta_B\\
 \theta_C&4\theta_B-8\theta_C&\theta_A-8\theta_B+10\theta_C
 \end{pmatrix}.
 \label{eq:clt-cavity-matrix}
\end{equation}
Then,  by \eqref{eq:clt-row-A}, \eqref{eq:clt-row-B} and \eqref{eq:clt-row-C}, we have
\begin{equation}
(g_A'(0),g_B'(0),g_C'(0))^{\mathsf T}=\mathcal M_1 \widetilde T(f)+O_f(N^{-1/2}).
 \label{eq:clt-g-prime-zero}
\end{equation}

Applying \eqref{eq:cavity-derivative-proof} twice implies that $g_E''(u)$ is a finite linear combination of terms of the form
\[
\sqrt N\,\nu_u\!\left[ \overline\varepsilon_{e(E)}\,f(X_{12})\,\Xi\, Q^-_{ab}Q^-_{cd} \right],
\]
where $\Xi$ is a product of last-spin variables and hence $|\Xi|\le1$. We therefore have, by \eqref{eq:clt-cavity-product},
\begin{equation}
 \sup_{0\le u\le1}|g_E''(u)|\le C_f \sqrt N\, \sup_{0\le u\le1}\nu_u\!\left[|Q^-_{ab}Q^-_{cd}|\right]\le O_f(N^{-1/2}).
 \label{eq:clt-second-cavity-derivative}
\end{equation}
Hence,
\begin{equation}
 T_E(f)=g_E(0)+g_E'(0)+O_f(N^{-1/2}).
 \label{eq:clt-cavity-Taylor}
\end{equation}
Similarly, \eqref{eq:cavity-derivative-proof} and \eqref{eq:clt-cavity-product} give
\begin{equation}
 \sqrt N\,\nu_0[Q^-_{e(E)}f(Y)]
 =
 \sqrt N\,\nu[Q^-_{e(E)}f(Y)]+O_f(N^{-1/2}).
 \label{eq:clt-u-endpoint-replacement}
\end{equation}
Since $Q_{e(E)}-Q^-_{e(E)}=N^{-1}\varepsilon_{e(E)}$,  \eqref{eq:clt-cavity-product} and 
\eqref{eq:clt-X-Y}  further imply
\begin{equation}
 \widetilde T_E(f)=T_E(f)+O_f(N^{-1/2}),
 \qquad E\in\{A,B,C\}.
 \label{eq:clt-Ttilde-T}
\end{equation}
Combining \eqref{eq:clt-g-zero}, \eqref{eq:clt-g-prime-zero},
\eqref{eq:clt-cavity-Taylor} and \eqref{eq:clt-Ttilde-T}, we obtain
\begin{equation}
 T(f)=\theta m_f+\mathcal M_1 T(f)+O_f(N^{-1/2}) .
 \label{eq:clt-vector-Stein}
\end{equation}

Define the mode correlations $U_f$, $V_f$ and $D_f$ by
\begin{align}
 U_f&\coloneqq T_A(f)-4T_B(f)+3T_C(f),
 \label{eq:clt-Uf}\\
 V_f&\coloneqq2T_B(f)-3T_C(f),
 \label{eq:clt-Vf}\\
 D_f&\coloneqq T_A(f)-2T_B(f)+T_C(f).
 \label{eq:clt-Df}
\end{align}
%The source modes\footnote{\color{red}What is this?} are $\kappa,\zeta,a$, respectively. 
Let $I$ be the $3$-dimensional identity matrix and let

$$
P\coloneqq
\begin{pmatrix}
1&-4&3\\
0&2&-3\\
1&-2&1
\end{pmatrix}.
$$
By the definitions of $\kappa$ and $\zeta$, a direct matrix multiplication gives
$$
P(I-\mathcal M_1)
=
\begin{pmatrix}
1-\beta^2\kappa&0&0\\
-\beta^2\zeta&1-\beta^2\kappa&0\\
0&0&1-\alpha
\end{pmatrix}P.
$$
Since
$$
PT(f)=
\begin{pmatrix}
U_f\\
V_f\\
D_f
\end{pmatrix},
\qquad
P\theta=
\begin{pmatrix}
\kappa\\
\zeta\\
 \alpha/\beta^2
\end{pmatrix},
$$
by \eqref{eq:clt-vector-Stein} we reach 
\begin{align}
 (1-\beta^2\kappa)U_f
 &=\kappa m_f+O_f(N^{-1/2}),
 \label{eq:clt-Stein-U}\\
 (1-\beta^2\kappa)V_f-\beta^2\zeta U_f
 &=\zeta m_f+O_f(N^{-1/2}),
 \label{eq:clt-Stein-V}\\
 (1-\alpha)D_f
 &=(\alpha/\beta^2) m_f+O_f(N^{-1/2}).
 \label{eq:clt-Stein-D}
\end{align}
In view of \eqref{eq:clt-positive-denominators}, we have
\begin{align*}
 U_f=\frac{\kappa}{1-\beta^2\kappa}m_f+O_f(N^{-1/2}),\quad 
 D_f&=\frac{\alpha}{\beta^2(1-\alpha)}m_f+O_f(N^{-1/2}).
 \end{align*}
Substitution in \eqref{eq:clt-Stein-V} yields
\begin{equation*}
 V_f=\frac{\zeta}{(1-\beta^2\kappa)^2}m_f+O_f(N^{-1/2}).
 \label{eq:clt-V-solved}
\end{equation*}
From the definitions of $U_f$, $V_f$, and $D_f$, we obtain
\begin{equation*}
T_A(f)=-V_f-2U_f+3D_f.
\end{equation*}
Indeed, this follows by solving the three linear relations defining $U_f$, $V_f$, and $D_f$ for $T_A(f)$.
 Hence
\begin{align}
 \nu[X_{12}f(X_{12})]
 &=
 \left(
 \frac{3\alpha}{\beta^2(1-\alpha)}
 -\frac{2\kappa}{1-\beta^2\kappa}
 -\frac{\zeta}{(1-\beta^2\kappa)^2}
 \right)\nu[f'(X_{12})]+O_f(N^{-1/2})
 \nonumber\\
 &=\sigma^2\nu[f'(X_{12})]+O_f(N^{-1/2}).
 \label{eq:clt-scalar-Stein}
\end{align}

% The estimate
% $\nu[|Q_{12}|^3]=o(N^{-1})$ (see \eqref{eq:Nnu3}) and
% \eqref{eq:cavity-remainders} show that the remainders in
% \eqref{eq:cavity-U}--\eqref{eq:cavity-D} are $o(N^{-1})$ at $s=1$.
% Thus
% \begin{equation}
%  NU\longrightarrow\frac{\kappa}{1-\beta^2\kappa}.
%  \label{eq:clt-second-U}
% \end{equation}
% The equations for $V$ and $D$ then give
% \begin{align}
%  NV&\longrightarrow\frac{\zeta}{(1-\beta^2\kappa)^2},
%  \,ND \longrightarrow\frac{a}{1-\alpha}.
%  \label{eq:clt-second-D}
% \end{align}
% Since $A=-V-2U+3D$,
% \begin{equation}
%  N\nu[Q_{12}^2]=NA\longrightarrow
%  \frac{3a}{1-\alpha}
%  -\frac{2\kappa}{1-\beta^2\kappa}
%  -\frac{\zeta}{(1-\beta^2\kappa)^2}
%  =\sigma^2.
%  \label{eq:clt-second-moment-limit}
% \end{equation}
% In particular,
% \begin{equation}
%  \sigma^2\ge0.
%  \label{eq:clt-variance-nonnegative}
% \end{equation}

Set
\begin{align}
 C_N(t)\coloneqq\nu[\exp(\sqrt{-1}\, tX_{12})].
 \label{eq:clt-SN}
\end{align}
Then
\begin{align}
 C_N'(t)&=\nu[\sqrt{-1} X_{12}\exp(\sqrt{-1}\,tX_{12})].
 \label{eq:clt-SN-derivative}
\end{align}
For fixed $t_0>0$, the $O_f(N^{-1/2})$ error in
\eqref{eq:clt-scalar-Stein} is uniform for
$f(x)=\exp(\sqrt{-1}\,tx)$ and $|t|\le t_0$.  Therefore
\begin{align}
 e^{-\sigma^2 t^2/2} (e^{\sigma^2 t^2/2} C_N(t))'= C_N'(t)+\sigma^2tC_N(t)
&=O_{t_0}(N^{-1/2}),
 \qquad |t|\le t_0.
 \label{eq:clt-C-approx-ODE}
 % e^{-\sigma^2 t^2/2} (e^{\sigma^2 t^2/2} S_N(t))' = S_N'(t)+\sigma^2tS_N(t)
 % &=O_{t_0}(N^{-1/2}),
 % \qquad |t|\le t_0.
 % \label{eq:clt-S-approx-ODE}
\end{align}
By  using 
$C_N(0)=1$ and integrating this over the interval $[0,t]$, we obtain
\begin{align*}
&e^{\sigma^2t^2/2}C_N(t) -1 =    \int_0^t (e^{\sigma^2s^2/2}C_N(s))' {\mathrm ds}= O_{t_0}(N^{-1/2}).
%&e^{\sigma^2t^2/2}S_N(t) =    \int_0^t (e^{\sigma^2s^2/2}S_N(s))' {\mathrm ds}= O_{t_0}(N^{-1/2}).
\end{align*}
Thus, we obtain that as $N\rightarrow \infty$
\begin{equation}
 \sup_{|t|\le t_0}
 \left|C_N(t)-e^{-\sigma^2t^2/2}\right|
 \longrightarrow0. 
 %\quad 
 %\sup_{|t|\le t_0}|S_N(t)|\longrightarrow0.  
 \label{eq:clt-C-uniform}
\end{equation}
Since $\sigma^2>0$, the function $e^{-\sigma^2t^2/2}$ is the characteristic
function of $\mathcal N(0,\sigma^2)$. Since the pointwise convergence of characteristic functions implies the weak convergence, we have  
\[
 \sqrt N\,(R_{12}-q)\xrightarrow{\mathrm d}\mathcal N(0,\sigma^2).
\]
\end{proof}

\begin{remark}[The divergence of $\sigma^2$ as approaching the AT line]\label{remark: sigma diverge}
 Let $(\beta_n,h_n)_{n=1}^\infty$ and $\beta,h>0$ satisfy $\alpha_n\uparrow 1$, $\beta_n\to \beta$, and $h_n\to h$. We write $q_n,r_n,\kappa_n,\zeta_n$ for the corresponding parameters. By definition (see \eqref{eq:alpha} and \eqref{eq:cavity-parameters}),  
$
1-\beta_n^2\kappa_n
=
(1-\alpha_n)+2\beta_n^2(q_n-r_n).
$ 
Moreover, since $q-r = \mathbb E[\tanh^2(\beta \sqrt{q} Z+h)\sech^2(\beta \sqrt{q} Z+h)]>0$, we obtain $\lim_{n\to \infty} (q_n-r_n)=q-r>0$. Thus, we have  
$$
1-\beta_n^2\kappa_n\to 2\beta^2(q-r)>0
$$
as $n \rightarrow \infty$.
Hence
$$
-\frac{2\kappa_n}{1-\beta_n^2\kappa_n}
-\frac{\zeta_n}{(1-\beta_n^2\kappa_n)^2}
=O(1).
$$
Therefore, since $1-2q_n+r_n\to 1-2q+r = \mathbb E[\sech^4(\beta \sqrt{q} Z+h)]>0$, 
$$
\sigma_n^2
=
\frac{3(1-2q_n+r_n)}{1-\alpha_n}+O(1)
\to+\infty.
$$

In the case $\beta=1$ and $h=0$, see \cite{DuHuangCriticalOverlap} which identifies the fluctuation order and its scaling limit. 
\end{remark}
\appendix

\section{Finite-step proof of the specialized Guerra--Talagrand bound}
\label{app:specialGT}

\subsection{A finite-step differentiation formula}

The following lemma isolates the only hierarchical Gaussian calculation
needed in the proof.

\begin{lemma}
\label{lem:finite-step-interpolation-derivative}
Let $\mathcal S$ be finite, let $\mu$ be a nonzero finite measure on
$\mathcal S$, and fix
\[
0=m_0\leq m_1\leq\cdots\leq m_M\leq m_{M+1}=1.
\]
Let $\{G_{k,t}(\xi):\xi\in\mathcal S,\ t\in[0,1]\}_{0\le k\le M}$ be independent centered Gaussian processes with almost surely $C^1$ sample paths in $t \in (0,1)$, and let $K_{k,t}$ denote their covariance kernels. Assume that $K_{k,t}$ is bounded.  Put
\[
\mathcal G_t(\xi) \coloneqq \sum_{k=0}^M G_{k,t}(\xi),
\qquad
C_t^{(j)}(\xi,\eta) \coloneqq \sum_{k=0}^jK_{k,t}(\xi,\eta).
\]
For a deterministic function $V:\mathcal S\to\mathbb R$, define
\[
X_M(t)=\log\sum_{\xi\in \mathcal S}
\mu(\xi )e^{\mathcal G_t(\xi)+V(\xi)}
\]
and, recursively for $j=M,\ldots,1$,
\begin{equation}\label{eq:defXB1}
X_{j-1}(t) \coloneqq
\begin{cases}
\displaystyle
\frac1{m_j}\log\E_j e^{m_jX_j(t)},&m_j>0,\\[3mm]
\E_jX_j(t),&m_j=0,
\end{cases}
\end{equation}
where $\E_j$ integrates only the field $G_{j,t}$.  For
$0\leq j\leq M$, set
\[
\mathcal I_jA \coloneqq \frac{\E_j[Ae^{m_jX_j(t)}]}{\E_je^{m_jX_j(t)}},
\]
and let
\[
p_t(\xi) \coloneqq
\frac{\mu(\xi)e^{\mathcal G_t(\xi)+V(\xi)}}
{\sum_{\eta\in\mathcal S}
\mu(\eta)e^{\mathcal G_t(\eta)+V(\eta)}}.
\]
For deterministic functions $F$ and $0\leq j\leq M$, define
\begin{align*}
\E\langle F(\xi) \rangle_t
&\coloneqq \sum_{\xi\in\mathcal S}F(\xi)
\mathcal I_0\cdots\mathcal I_Mp_t(\xi),\\
\E\langle F(\xi,\eta)\rangle_{j,t}
&\coloneqq \sum_{\xi,\eta\in\mathcal S}F(\xi,\eta)
\mathcal I_0\cdots\mathcal I_j\left[
\bigl(\mathcal I_{j+1}\cdots\mathcal I_Mp_t(\xi)\bigr)
\bigl(\mathcal I_{j+1}\cdots\mathcal I_Mp_t(\eta)\bigr)
\right],
\end{align*}
where an empty operator product is the identity.  Then, for $0<t<1$,
\begin{equation}
\begin{aligned}
\frac{\mathrm d}{\mathrm dt} \E_0X_0(t)
={}&\frac1{2}\E\left\langle
\dot C_t^{(M)}(\xi,\xi)
\right\rangle_t-\frac1{2}\sum_{j=0}^M(m_{j+1}-m_j)
\E\left\langle
\dot C_t^{(j)}(\xi,\eta)
\right\rangle_{j,t}
\end{aligned}
\label{eq:finite-step-interpolation-derivative}
\end{equation}
where $\dot C_t^{(j)}(\xi,\eta) \coloneqq \frac{\partial}{\partial t} C_t^{(j)}(\xi,\eta)$.
\end{lemma}

\begin{proof}
Throughout the proof, for fixed $t$, we regard $X_j(t)$ as a smooth function of the Gaussian coordinates
\[
\{G_{k,t}(\xi):0\leq k\leq j,\ \xi\in\mathcal S\}
\]
that have not yet been integrated out; the fields at levels
$j+1,\ldots,M$ have already been averaged out by the recursion.
Thus $\partial_xX_{j-1}=\mathcal I_j(\partial_xX_j)$ whenever $x$ is
a Gaussian coordinate not integrated by $\E_j$ or the time coordinate $t$.  Indeed, if
$Z_j=\E_je^{m_jX_j}$ and $m_j>0$, then
\begin{align}
\partial_xX_{j-1}
&=\frac1{m_j}\frac{\partial_xZ_j}{Z_j}
=\frac{\E_j[e^{m_jX_j}\partial_xX_j]}{Z_j}
=\mathcal I_j(\partial_xX_j).
\label{eq:finite-step-first-chain-rule}
\end{align}
The same identity for $m_j=0$ follows by differentiating
$X_{j-1}=\E_jX_j$.
From \eqref{eq:finite-step-first-chain-rule} it follows that
\begin{equation}\label{eq:dX=IP}
\partial_{G_{k,t}(\xi)}X_k =\mathcal I_{k+1}\cdots\mathcal I_Mp_t(\xi).
\end{equation}

For any differentiable $A$, the quotient rule gives the more general
identity
\begin{align}
\partial_y\mathcal I_jA
={}&\mathcal I_j(\partial_yA)
+m_j\mathcal I_j(A\partial_yX_j)
-m_j\mathcal I_j(A)\mathcal I_j(\partial_yX_j).
\label{eq:derivative-tilted-expectation}
\end{align}
Consequently,
\begin{align}
\partial_{xy}X_{j-1}
={}&\mathcal I_j(\partial_{xy}X_j)
+m_j\left\{
\mathcal I_j(\partial_xX_j\partial_yX_j)
-\mathcal I_j(\partial_xX_j)
 \mathcal I_j(\partial_yX_j)
\right\}.
\label{eq:finite-step-second-chain-rule}
\end{align}
Again, this includes $m_j=0$, since the last two terms then vanish.

By \eqref{eq:finite-step-first-chain-rule}, we have
\begin{equation}
\begin{aligned}
    \frac{\mathrm d}{\mathrm dt}\E X_0 
&=\mathcal{I}_0\dots \mathcal{I}_{M}\frac{\mathrm d}{\mathrm dt} X_{M}(t)\\
&=\sum_{k=0}^M\mathcal{I}_0\dots \mathcal{I}_{M}\left( \sum_{\xi\in \mathcal{S}} p_t(\xi) \dot{G}_{k,t}(\xi)\right)\\
&= \sum_{k=0}^M \sum_{\xi\in \mathcal{S}} \mathcal{I}_0\dots\mathcal{I}_k\left(\dot{G}_{k,t}(\xi) (\mathcal{I}_{k+1} \dots\mathcal{I}_{M} (p_t(\xi)) )\right).
\label{eq:gaussian-covariance-differentiation}
\end{aligned}
\end{equation}
When $m_k>0$, by Gaussian integration by parts (see \cite[Vol. 1, Sec. A.4]{Talagrand}), i.e. for centered jointly  Gaussian random variables  $X,Y_i$,  
$$\mathbb E[Xf(Y_1\dots Y_d)] = 
\sum_{i=1}^d \operatorname{Cov}(X,Y_i)\mathbb E[\partial_i f(Y)],$$ 
where $\operatorname{Cov}(\dot{G}_{k,t}(\xi),{G}_{k,t}(\eta))+\operatorname{Cov}(\dot{G}_{k,t}(\eta),{G}_{k,t}(\xi)) = \dot K_{k,t}(\xi,\eta)$, by \eqref{eq:dX=IP} we have
\begin{align*}
\sum_{\xi\in \mathcal S} \mathcal{I}_k\left(\dot{G}_{k,t}(\xi) (\mathcal{I}_{k+1} \dots\mathcal{I}_{M} p_t(\xi) )\right)& = \sum_{\xi\in \mathcal S} Z_k^{-1}  \mathbb E_k\left[e^{m_k X_k} \dot{G}_{k,t}(\xi) (\mathcal{I}_{k+1} \dots\mathcal{I}_{M} p_t(\xi) )\right] \\
 & =\frac{1}{2} \sum_{\xi,\eta\in \mathcal{S}} \dot K_{k,t}(\xi,\eta) (Z_k)^{-1}\mathbb E_k\left[\partial_{{G}_{k,t}(\eta)}  ({e^{m_k X_k}} (\mathcal{I}_{k+1} \dots\mathcal{I}_{M} p_t(\xi) )) \right]\\
 &=\frac{1}{2}\sum_{\xi,\eta\in \mathcal{S}} \dot K_{k,t}(\xi,\eta) (m_k Z_k)^{-1}\mathbb E_k\left[\partial_{{G}_{k,t}(\eta)} \partial_{{G}_{k,t}(\xi)} {e^{m_k X_k}}\right] .
\end{align*}
Let $H_{j,t}(\xi)=\sum^j_{r=0} G_{r,t}(\xi)$.   For fixed $k$, write
\[
u_{k}^\xi=\partial_{H_{k,t}(\xi)}X_k,
\qquad
h_{k}^{\xi\eta}
=\partial_{H_{k,t}(\xi)}\partial_{H_{k,t}(\eta)}X_k.
\]
The contribution obtained by differentiating only the
covariance of $G_{k,t}$ is
\begin{align*}
(m_k Z_k)^{-1}\E_k\!\left[
\partial_{G_k(\xi)}\partial_{G_k(\eta)}e^{m_kX_k}
\right]
&= 
\mathcal I_k\left(
h_k^{\xi\eta}+m_ku_k^\xi u_k^\eta
\right).
\end{align*}
For $m_k=0$, the same formula follows directly from
\eqref{eq:gaussian-covariance-differentiation}, with the second summand
equal to zero.  At $k=0$ it also holds because
$\mathcal I_0=\E_0$ and $m_0=0$. 
Hence, letting
\[
\mathcal P_k=\mathcal I_0\mathcal I_1\cdots\mathcal I_k,
\qquad
Q_k^{\xi\eta}=h_k^{\xi\eta}+m_ku_k^\xi u_k^\eta,
\]
by \eqref{eq:gaussian-covariance-differentiation} we have
\begin{equation}
\frac{\mathrm d}{\mathrm dt}\E X_0
=\frac12\sum_{k=0}^M\sum_{\xi,\eta\in\mathcal S}
\dot K_{k,t}(\xi,\eta)\mathcal P_kQ_k^{\xi\eta}.
\label{eq:finite-step-level-contributions}
\end{equation}

It remains to compute $\mathcal P_kQ_k^{\xi\eta}$.  At the terminal
level, direct differentiation gives
\[
u_M^\xi=p_t(\xi),\qquad
h_M^{\xi\eta}
=\mathbf 1_{\{\xi=\eta\}}p_t(\xi)-p_t(\xi)p_t(\eta).
\]
Since $m_{M+1}=1$, it follows that
\begin{equation}\label{eq:QM}
Q_M^{\xi\eta}
=\mathbf 1_{\{\xi=\eta\}}p_t(\xi)
-(m_{M+1}-m_M)u_M^\xi u_M^\eta.
\end{equation}
Note that \eqref{eq:dX=IP} implies $\partial_{H_{j-1,t}(\xi)}X_j = \partial_{H_{j,t}(\xi)}X_j$.
For $j=M,\ldots,1$, equations
\eqref{eq:finite-step-first-chain-rule} and
\eqref{eq:finite-step-second-chain-rule} yield
\begin{align*}
u_{j-1}^\xi
&=\mathcal I_ju_j^\xi, \quad 
h_{j-1}^{\xi\eta}
=\mathcal I_jh_j^{\xi\eta}
+m_j\left(
\mathcal I_j(u_j^\xi u_j^\eta)
-u_{j-1}^\xi u_{j-1}^\eta
\right),
\end{align*}
and hence
\[
Q_{j-1}^{\xi\eta}=h_{j-1}^{\xi\eta}+m_{j-1}u_{j-1}^\xi u_{j-1}^\eta
=\mathcal I_jQ_j^{\xi\eta}
-(m_j-m_{j-1})u_{j-1}^\xi u_{j-1}^\eta.
\]
By applying $\mathcal P_{j-1}$ to both sides of this equality, we have
\[
\mathcal P_{j-1}Q_{j-1}^{\xi\eta}
=\mathcal P_jQ_j^{\xi\eta}
-(m_j-m_{j-1})
\mathcal P_{j-1}(u_{j-1}^\xi u_{j-1}^\eta).
\]
Iterating from $j=M$ down to $j=k+1$ and using \eqref{eq:QM}
gives
\begin{align}
\mathcal P_kQ_k^{\xi\eta}
={}&\mathcal P_M
\left(\mathbf 1_{\{\xi=\eta\}}p_t(\xi)\right)
-\sum_{j=k}^M(m_{j+1}-m_j)
\mathcal P_j(u_j^\xi u_j^\eta).
\label{eq:finite-step-telescoped-hessian}
\end{align}
Note that similarly to \eqref{eq:dX=IP} it holds that $u_j^\xi =\mathcal I_{j+1}\cdots\mathcal I_Mp_t(\xi)$.
Substituting \eqref{eq:finite-step-telescoped-hessian} into
\eqref{eq:finite-step-level-contributions} therefore gives 
\begin{align*}
   \frac{\mathrm d}{\mathrm dt}\E X_0 
={}&\frac12\E\left\langle
\sum_{k=0}^M\dot K_{k,t}(\xi,\xi)
\right\rangle_t-\frac12\sum_{k=0}^M\sum_{j=k}^M(m_{j+1}-m_j)
\E\left\langle
\dot K_{k,t}(\xi,\eta)
\right\rangle_{j,t}\\
={}&\frac12\E\left\langle
\dot C_t^{(M)}(\xi,\xi)
\right\rangle_t
-\frac12\sum_{j=0}^M(m_{j+1}-m_j)
\E\left\langle
\dot C_t^{(j)}(\xi,\eta)
\right\rangle_{j,t},
\end{align*}
where we have used $\sum_{k=0}^M\sum_{j=k}^M\dot K_{k,t}(\xi,\eta)=\sum_{j=0}^M\sum_{k=0}^j \dot K_{k,t}(\xi,\eta)=\sum_{j=0}^M \dot C_t^{(j)}(\xi,\eta)$.
\end{proof}

\subsection{Application to the two-replica path}
Fix $N\geq1$, $\beta>0$, $h\in\mathbb R$, $s\in[0,1]$, and $q\in(0,1)$.
We use the same notation as in Section \ref{subsec:GT}. Recall that
\[
\xi_s(r)
=
\beta^2(1-s)qr+\frac{s\beta^2}{2}r^2.
\]
Let $H_{N,s}^{\mathrm{cen}}$ be a centered Gaussian field on
$\Sigma_N$ normalized by
\begin{equation}
\E H_{N,s}^{\mathrm{cen}}(\sigma)
H_{N,s}^{\mathrm{cen}}(\tau)
=N\xi_s(R(\sigma,\tau)).
\label{eq:appendix-exact-covariance}
\end{equation}
Set
\[
H_{N,s}(\sigma)=H_{N,s}^{\mathrm{cen}}(\sigma)
+h\sum_{i=1}^N\sigma_i,\quad 
\mathcal R_N=\{-1,-1+2/N,\ldots,1\}.
\]
For $v\in\mathcal R_N$, let
\begin{equation}
Z_{N,s}^{(2)}(v)=
\sum_{\substack{\sigma^1,\sigma^2\in\Sigma_N\\
R(\sigma^1,\sigma^2)=v}}
e^{H_{N,s}(\sigma^1)+H_{N,s}(\sigma^2)}.
\label{eq:appendix-constrained-partition}
\end{equation}

We prove \eqref{eq:GTfunctionalbound} as follows.

\begin{lemma}
\label{lem:specialGT}
For every $N\geq1$ and $v\in\mathcal R_N$,
\begin{equation}
\frac1N\E\log Z_{N,s}^{(2)}(v)
\leq\inf_{\lambda\in\mathbb R}\mathfrak P_s(\lambda,v).
\label{eq:appendix-GT-infimum}
\end{equation}
\end{lemma}

\begin{proof}
%Fix $\lambda\in\mathbb R$.  The monotonicity of the levels required
%by Lemma~\ref{lem:finite-step-interpolation-derivative}, as well as the
%positivity of all Gaussian variances below, was verified in
%\eqref{eq:finite-step-vector-increments} and
%\eqref{eq:finite-step-scalar-increments}.
In view of \eqref{eq:finite-step-vector-increments} and \eqref{eq:finite-step-scalar-increments}, we are able to choose independent Gaussian vectors $z_{i,j}\in\mathbb R^2$ and
independent Gaussian variables $y_j$ with
\begin{equation}\label{eq:covzy}
\E z_{i,j}z_{i,j}^{\mathsf T}=B_j-B_{j-1},
\qquad
\E y_j^2=d_j-d_{j-1}.
\end{equation}
Let $Y_{i,0}$ be independent copies of
$Y_0=h+\beta\sqrt{(1-s)q}\,Z$. Since
\[
B_0=\beta^2(1-s)q
\begin{pmatrix}1&1\\1&1\end{pmatrix},
\]
the vector $(Y_{i,0}-h,Y_{i,0}-h)$ has the covariance $B_0$.
Write
\[
{\mathcal{S}}_N(v)=\{(\sigma^1,\sigma^2)\in\Sigma_N^2:
R(\sigma^1,\sigma^2)=v\}.
\]
For $\boldsymbol\sigma=(\sigma^1,\sigma^2)\in{\mathcal{S}}_N(v)$, introduce
the independent Gaussian random variables $\{ G_{j,t}(\boldsymbol\sigma)\} _{j=0}^M$ with parameter $t \in [0,1]$ as
\begin{align*}
G_{0,t}(\boldsymbol\sigma)
\coloneqq {}&\sqrt t\bigl(H_{N,s}^{\mathrm{cen}}(\sigma^1)
+H_{N,s}^{\mathrm{cen}}(\sigma^2)\bigr)+\sqrt{1-t}\sum_{i=1}^N\sum_{a=1}^2
(Y_{i,0}-h)\sigma_i^a,\\
G_{j,t}(\boldsymbol\sigma)
\coloneqq {}&\sqrt{1-t}\sum_{i=1}^N\sum_{a=1}^2z_{i,j}^a\sigma_i^a
+\sqrt{tN}\,y_j,
\qquad 1\leq j\leq M,
\end{align*}
and the deterministic term
\[
V(\boldsymbol\sigma) \coloneqq h\sum_{i=1}^N(\sigma_i^1+\sigma_i^2)
+\lambda\sum_{i=1}^N\sigma_i^1\sigma_i^2-\lambda Nv.
\]
Let
\[
X_M(t) \coloneqq \log\sum_{\boldsymbol\sigma\in{\mathcal{S}}_N(v)}
\exp\left(\sum_{j=0}^MG_{j,t}(\boldsymbol\sigma)
+V(\boldsymbol\sigma)\right),
\]
define $\{ X_j(t)\} _{j=0}^M$ as in Lemma~\ref{lem:finite-step-interpolation-derivative},
and set
\[
\varphi_N(t) \coloneqq \frac1N\E X_0(t).
\]

We first compute $\varphi_N(t)$ at the endpoints $t=0,1$. 
Note that
\begin{align*}
X_M(0) &= \log\sum_{\boldsymbol\sigma\in{\mathcal{S}}_N(v)} \exp\left( \sum _{i=1}^N Y_{i,0}(\sigma _i^1 + \sigma _i^2) + \sum _{j=1}^M \sum _{i=1}^N (z_{i,j}^1 \sigma _i^1 + z_{i,j}^2 \sigma _i^2 ) +\lambda\sum_{i=1}^N\sigma_i^1\sigma_i^2-\lambda Nv \right) \\
&\leq \log \sum_{\boldsymbol\sigma \in\Sigma_N^2} \exp\left( \sum _{i=1}^N Y_{i,0}(\sigma _i^1 + \sigma _i^2) + \sum _{j=1}^M \sum _{i=1}^N (z_{i,j}^1 \sigma _i^1 + z_{i,j}^2 \sigma _i^2 ) +\lambda\sum_{i=1}^N\sigma_i^1\sigma_i^2-\lambda Nv \right) \\
&= \log \prod _{i=1}^N \sum _{\varepsilon _1, \varepsilon _2 = \pm 1}  \exp\left( \left( \varepsilon _1 \left( Y_{0,i} + \sum _{j=1}^M z_{i,j}^1 \right) + \varepsilon _2 \left( Y_{0,i} + \sum _{j=1}^M z_{i,j}^2 \right) + \lambda \varepsilon _1 \varepsilon _2 \right) \right) -\lambda N v \\
&= 2 N \log 2 + \sum _{i=1}^N f_\lambda \left( Y_{0,i} + \sum _{j=1}^M z_{i,j}^1 , Y_{0,i} + \sum _{j=1}^M z_{i,j}^2 \right) -\lambda N v .
\end{align*}
Since $G_j$ in Section \ref{subsec:GT} has the same law as $z_{i,j}$ (see \eqref{eq:covzy}), by the monotonicity of the map ${\mathcal T}_{m,C}$ defined in \eqref{eq:GTsemigroup}, similarly to \eqref{eq:relationXU} we obtain
\[
X_0(0) \leq 2 N \log 2 + \sum _{i=1}^N U_s^{\lambda ,v} (Y_{0,i}, Y_{0,i}) -\lambda N v .
\]
This inequality and the independence of the variables $Y_{0,i}$ yield 
\begin{equation}
\varphi_N(0)
\leq
2\log2+\E U_s^{\lambda,v}(Y_0,Y_0)-\lambda v.
\label{eq:GT-endpoint-zero}
\end{equation}

We consider  the case $t=1$.
An explicit calculation implies
\begin{align*}
X_M(1) &= \log \sum_{\boldsymbol\sigma\in{\mathcal{S}}_N(v)} \exp \left( H_{N,s}(\sigma ^1) + H_{N,s}(\sigma ^2) + \sqrt{N} \sum _{j=1}^M y_j\right) \\
&= \log Z_{N,s}^{(2)}(v) + \sqrt{N} \sum _{j=1}^M y_j,
\end{align*}
where $Z_{N,s}^{(2)}(v)$ is defined in \eqref{eq:appendix-constrained-partition}. 
Since for each $x\in {\mathbb R}$ it holds that
\[
\begin{cases}
\displaystyle
\frac1{m_j}\log\E e^{m_j(x+\sqrt N\,y_j)}
=x+\frac{Nm_j}{2}(d_j-d_{j-1}),&m_j>0,\\[3mm]
\displaystyle
\E (x+\sqrt N\,y_j)=x,&m_j=0 ,
\end{cases}
\]
by \eqref{eq:defXB1} we have
\begin{align}
\varphi_N(1)
&=\frac1N\E\log Z_{N,s}^{(2)}(v)+\frac12\sum_{j=1}^Mm_j(d_j-d_{j-1})
\notag\\
&=\frac1N\E\log Z_{N,s}^{(2)}(v)+\mathcal L_s(v),
\label{eq:GT-endpoint-one}
\end{align}
where $\mathcal L_s(v)$ is defined in \eqref{eq:defLv}.

We calculate the derivative of $\varphi _N$.
Denote the covariance of $\sum _{i=0}^j G_{i,t} (\boldsymbol\sigma)$ and $\sum _{i=0}^j G_{i,t} (\boldsymbol\tau)$ by $C_t^{(j)}(\boldsymbol\sigma,\boldsymbol\tau)$.
Since \eqref{eq:appendix-exact-covariance} and $\E (Y_{0,i} - h)^2 = \beta ^2 (1-s) q$ imply
\[
\E \left[ G_{0,t} (\boldsymbol\sigma) G_{0,t} (\boldsymbol\tau)\right] = tN\sum_{k,\ell=1}^2\xi_s (R(\sigma^k,\tau^\ell)) + (1-t) \beta ^2 (1-s) q N \sum_{k,\ell=1}^2 R(\sigma^k,\tau^\ell)
\]
and \eqref{eq:covzy} implies
\begin{align*}
\E \left[ G_{j,t} (\boldsymbol\sigma) G_{j,t} (\boldsymbol\tau)\right] &= (1-t) \sum _{i=1}^N \sum_{k,\ell=1}^2 \sigma _i ^k \tau_i^{\ell} \E [z_{i,j}^k z_{i,j}^\ell ] + tN \E y_j^2 \\
&= (1-t) \sum _{i=1}^N \boldsymbol\sigma _i ^{\mathsf T} (B_j -B_{j-1}) \boldsymbol\tau _i + tN (d_j -d_{j-1}),
\end{align*}
by \eqref{eq:defB} and \eqref{eq:defd} we have
\begin{align*}
C_t^{(j)}(\boldsymbol\sigma,\boldsymbol\tau)&={} tN\sum_{k,\ell=1}^2\xi_s (R(\sigma^k,\tau^\ell)) + (1-t) \beta ^2 (1-s) q N \sum_{k,\ell=1}^2 R(\sigma^k,\tau^\ell) \\
&\quad + (1-t) \sum _{i=1}^N \boldsymbol\sigma _i ^{\mathsf T} (B_j -B_0) \boldsymbol\tau _i + tN (d_j -d_0)\\
&={}tN\sum_{k,\ell=1}^2\xi_s(R(\sigma^k,\tau^\ell))
+(1-t)N\sum_{k,\ell=1}^2\xi_s'(Q_{j,k\ell}^v)R(\sigma^k,\tau^\ell)
 +\frac{s\beta ^2 tN}{2}\sum_{k,\ell=1}^2 (Q_{j,k\ell}^v)^2\\
&=N\sum_{k,\ell=1}^2\xi_s'(Q_{j,k\ell}^v)R(\sigma^k,\tau^\ell)+ t\left(\frac{s \beta^2 N}{2}\sum_{k,\ell=1}^2 (R(\sigma^k,\tau^\ell) -Q_{j,k\ell}^v )^2\right),
\end{align*}
due to $\xi_s(x)-\xi_s(y)-\xi_s'(y)(x-y)=\frac{s\beta^2}{2}(x-y)^2$. 
Since for $\boldsymbol\sigma\in{\mathcal{S}}_N(v)$
\[
\bigl(R(\sigma^k,\sigma^\ell )\bigr)_{k,\ell =1}^2
=\begin{pmatrix}1&v\\v&1\end{pmatrix}=Q_M^v,
\]
it follows that
\[
\dot C_t^{(M)}(\boldsymbol\sigma,\boldsymbol\sigma)
=\frac{s\beta^2 N}{2}\sum_{k,\ell=1}^2
(R(\sigma^k,\sigma^\ell) -Q_{M,k\ell}^v)^2=0.
\]
Hence, Lemma~\ref{lem:finite-step-interpolation-derivative} 
now yields
\begin{equation}
\varphi_N'(t)
=-\frac{s\beta^2}{4}
\sum_{j=0}^M(m_{j+1}-m_j)
\E\left\langle
\sum_{k,\ell=1}^2(R(\sigma^k,\tau^\ell) -Q_{j,k\ell}^v)^2
\right\rangle_{j,t}
\leq0.
\label{eq:special-GT-interpolation-derivative}
\end{equation}
Combining \eqref{eq:GT-endpoint-zero},
\eqref{eq:GT-endpoint-one}, and
\eqref{eq:special-GT-interpolation-derivative}, and taking the infimum over $\lambda$ proves
\eqref{eq:appendix-GT-infimum}.
\end{proof}

\section{Price's theorem}\label{sec:Price}

\begin{theorem}[Price's theorem \cite{Price}]\label{thm:Price}
For each $t\in {\mathbb R}$, let $X(t)$ be an $n$-dimensional centered Gaussian vector with covariance matrix $C(t) = (C_{ij}(t))_{i,j =1}^n$. 
Assume that $C_{ij}(t)$ is continuously differentiable for all $i,j =1,2, \dots ,n$. 
Then, for any $f\in C_b^2 ({\mathbb R}^n)$, it holds that % \footnote{\textcolor{red}{SN: The theorem permits singular covariance matrices, but the proof below uses $\det C(t)$ and $C(t)^{-1}$. It needs an approximation $C\mapsto C+\varepsilon I$ followed by $\varepsilon\downarrow0$, or a singular-covariance proof via characteristic functions.}}
\[
\frac{\mathrm d}{\mathrm dt} \E f(X(t)) = \frac{1}{2} \sum _{i,j =1}^n \left( \frac{\mathrm d}{\mathrm dt} C_{i,j}(t)\right) \E \left[ \frac{\partial ^2 f}{\partial x_i \partial x_j} (X(t)) \right] .
\]
\end{theorem}

\begin{proof}
It suffices to prove the claim for compactly supported $f\in C_b^2({\mathbb R}^n)$.  
Let ${\mathcal F}g$ be the Fourier transform of $g$, i.e.
\[
({\mathcal F} g) (\xi ) = \frac{1}{(2\pi )^{n/2}}\int _{\mathbb R^n} g(x) e^{-\sqrt{-1}x\cdot \xi } dx, \quad \xi \in {\mathbb R}^n ,
\]
and let $C^\varepsilon (t) := \varepsilon I + C(t)$ where $\varepsilon >0$ and $I$ is the unit matrix.
Denote an $n$-dimensional centered Gaussian vector with a covariance matrix $C^\varepsilon (t)$ by $X^\varepsilon (t)$ for $t\in {\mathbb R}$.
Then, by the Plancherel equality we have
\begin{align*}
\E f(X^\varepsilon (t)) &= \int _{{\mathbb R}^n} f(x) \cdot \frac{1}{\sqrt{(2\pi)^n \det C^\varepsilon(t)}} \exp \left( -\frac{1}{2}\langle x , C^\varepsilon (t)^{-1} x \rangle \right) dx \\
&= \frac{1}{(2\pi )^{n/2}} \int _{{\mathbb R}^n} ({\mathcal F}f)(\xi ) \exp \left( -\frac{1}{2} \langle \xi , C^\varepsilon (t)\xi \rangle \right) d\xi,
\end{align*}
where $\langle \cdot ,\cdot \rangle$ denotes the inner product on ${\mathbb R}^n$ and $\det C^\varepsilon(t)$ denotes the determinant of $C^\varepsilon(t)$.  Since $X^\varepsilon(t)$ converges in distribution to $X(t)$ and $C^\varepsilon(t)$ converges componentwise to $C(t)$ as $\varepsilon\downarrow0$, differentiating with respect to $t$ and then passing to the limit $\varepsilon\downarrow0$ yields  
\begin{align*}
\frac{\mathrm d}{\mathrm dt} \E f(X(t)) &= - \frac{1}{2(2\pi )^{n/2}}\sum _{i,j =1}^n \left( \frac{\mathrm d}{\mathrm dt} C_{i,j}(t)\right) \int _{{\mathbb R}^n} ({\mathcal F}f)(\xi ) \xi _i \xi _j \exp \left( -\frac{1}{2} \langle \xi , C(t)\xi \rangle \right) d\xi,
\end{align*}
where $\xi = (\xi _1, \xi _2 ,\dots , \xi _n)$.
Applying the fact
\[
\xi _i \xi _j ({\mathcal F}f)(\xi ) = - {\mathcal F} \left( \frac{\partial ^2 f}{\partial x_i \partial x_j} \right)(\xi ) ,
\]
to this equality, we obtain
\begin{equation}\label{eq:Price01}
\frac{\mathrm d}{\mathrm dt} \E f(X(t)) = \frac{1}{2(2\pi )^{n/2}}\sum _{i,j =1}^n \left( \frac{\mathrm d}{\mathrm dt} C_{i,j}(t)\right) \int _{{\mathbb R}^n} {\mathcal F} \left( \frac{\partial ^2 f}{\partial x_i \partial x_j} \right)(\xi ) \exp \left( -\frac{1}{2} \langle \xi , C(t)\xi \rangle \right) d\xi .
\end{equation}
Similarly to the above, it holds that
\begin{align*}
&\frac{1}{(2\pi )^{n/2}} \int _{{\mathbb R}^n} {\mathcal F} \left( \frac{\partial ^2 f}{\partial x_i \partial x_j} \right)(\xi ) \exp \left( -\frac{1}{2} \langle \xi , C(t)\xi \rangle \right) d\xi \\
&= \lim _{\varepsilon \downarrow 0} \int _{{\mathbb R}^n} \left( \frac{\partial ^2 f}{\partial x_i \partial x_j} \right) (x) \cdot \frac{1}{\sqrt{(2\pi)^n \det C^\varepsilon(t)}} \exp \left( -\frac{1}{2}\langle x , C^\varepsilon (t)^{-1} x \rangle \right) dx\\
&= \E \left[ \frac{\partial ^2 f}{\partial x_i \partial x_j} (X(t)) \right] .
\end{align*}
This equality and \eqref{eq:Price01} yield the assertion.
\end{proof}

\section*{Acknowledgements}

S.K.\ acknowledges support from JSPS KAKENHI Grant Numbers JP22H00099 and JP23K20801.
S.N.\ acknowledges support from JSPS KAKENHI Grant Numbers JP24K16937 and JP25K00911.

\section*{Statement on the Use of Artificial Intelligence}
Shortly after posting \cite{KusuokaNakajima1}, inspired by the arguments in \cite{TalagrandParisi}, the authors realized that the flatness of the Guerra--Talagrand one-step  RSB bound near the replica-symmetric point is sufficient to obtain a preliminary weak concentration estimate for the overlap, as explained in Subsection~\ref{subsec:proof-outline}. This weak concentration can be upgraded, using Talagrand's cavity method, to the optimal $N^{-1}$ bound. Thus, the main remaining task was to establish the required flatness of the one-step RSB bound.

To investigate this flatness,  we asked ChatGPT 5.6 Sol to construct a proof, with reference to Lopatto's paper \cite{Lopatto}. After several iterations, ChatGPT produced an argument that led to the proof in this paper. The authors subsequently revised the argument in detail and take full responsibility for its correctness. We also used Codex to formalize the proof in Lean 4 \cite{brabra}.

\end{document}